%% file: Main.tex
\RequirePackage[l2tabu,orthodox]{nag}		% for deprecated commands
\documentclass[reqno,letterpaper]{article}						% for NIPS
\PassOptionsToPackage{numbers,sort&compress}{natbib}
\usepackage[numbers,sort&compress]{natbib}
\usepackage[final]{neurips_2019}

\usepackage{amsmath}		% for AMS macros
\usepackage{amssymb}		% for AMS symbols
\usepackage{amsfonts}		% for AMS fonts
\usepackage{amsthm}		% for theorems

\usepackage{mathtools}		% for advanced math
\mathtoolsset{%
}

\usepackage[utf8]{inputenc}		% for source encoding
\usepackage[T1]{fontenc}		% for font encoding
\renewcommand\bf{\bfseries}

\usepackage{dsfont}		% for blackboard bold font
\usepackage{inconsolata}

\usepackage[%		% for math font selection
cal=cm,
]
{mathalfa}

\usepackage[labelfont={bf,small},labelsep=colon,font=small, singlelinecheck=false]{caption}	% for caption control
\usepackage[dvipsnames,svgnames]{xcolor}		% for color
\colorlet{MyBlue}{DodgerBlue!75!Black}
\colorlet{MyGreen}{DarkGreen!85!Black}

\newcommand{\Paragraph}[1]{\paragraph{#1}}
\newcommand{\afterhead}{.}		% for changing headings
\usepackage{placeins}

\usepackage{wasysym}		% for extra symbols

\usepackage{fancybox}		% for fancy frames
\usepackage{subcaption}		% for subfigures
\usepackage{tikz}		% for figures
\usetikzlibrary{calc,patterns}

\usepackage{array}		% for flexible tables
\newcolumntype{C}[1]{>{\centering\let\newline\\\arraybackslash\hspace{0pt}}m{#1}}
\usepackage{booktabs}		% for better tables
\usepackage{paralist}		% for inline lists

\usepackage{enumitem}
\usepackage{multirow}
\usepackage{threeparttable}   % caption width not exceed table

\usepackage{microtype}		% for microtypography

\usepackage{acronym}		% for acronyms
\usepackage{latexsym}		% for symbols
\usepackage{xfrac}		% for better fractions
\usepackage{xspace}		% for flexible spaces
\usepackage{xparse}     % for two optional arguments

\usepackage[numbers,sort&compress]{natbib}		% for citations

\usepackage{hyperref}
\hypersetup{
colorlinks=true,
linktocpage=true,
pdfstartview=FitH,
breaklinks=true,
pdfpagemode=UseNone,
pageanchor=true,
pdfpagemode=UseOutlines,
plainpages=false,
bookmarksnumbered,
bookmarksopen=false,
bookmarksopenlevel=1,
hypertexnames=true,
pdfhighlight=/O,
urlcolor=Maroon,linkcolor=MyBlue!60!black,citecolor=DarkGreen!70!black,	% for on-screen
pdftitle={},
pdfauthor={},
pdfsubject={},
pdfkeywords={},
pdfcreator={pdfLaTeX},
pdfproducer={LaTeX with hyperref}
}

\usepackage[sort&compress,capitalize,nameinlink]{cleveref}		% for cleveref formatting
\crefname{assumption}{Assumption}{Assumptions}
\usepackage[textwidth=30mm,disable]{todonotes}		% for todo's
\newcommand{\debug}[1]{#1}		% for removing macro coloring
\usepackage{algorithm}		% for algorithm environments
\usepackage{algpseudocode}		% for algorithm macros
\usepackage{thmtools,thm-restate, thm-autoref}

\theoremstyle{plain}
\newtheorem{theorem}{Theorem}		% for theorems
\newtheorem{lemma}{Lemma}		% for lemmas
\newtheorem{proposition}{Proposition}		% for propositions

\newtheorem*{corollary*}{Corollary}		% for corollaries (unnumbered)

\newtheorem*{lemma*}{Lemma}

\theoremstyle{definition}
\newtheorem{definition}[theorem]{Definition}		% for definitions
\newtheorem{assumption}{Assumption}		% for assumptions
\newtheorem{example}{Example}		% for examples

\newtheorem*{definition*}{Definition}		% for definitions (unnumbered)
\newtheorem*{assumption*}{Assumptions}		% for assumptions (unnumbered)

\makeatletter		% for custom tags
\newcommand{\asmtag}[1]{% \asmtag{<tag>}
  \let\oldtheassumption\theassumption% Store \theassumption
  \renewcommand{\theassumption}{#1}% Redefine it to a fixed value
  \g@addto@macro\endassumption{% At \end{assumption}, ...
    \addtocounter{assumption}{-1}% ...restore assumption counter value and...
    \global\let\theassumption\oldtheassumption}% ...restore \theassumption
  }
\makeatother

\theoremstyle{remark}
\newtheorem{remark}{Remark}		% for remarks

\newtheorem*{remark*}{Remark}		% for remarks (unnumbered)
\newtheorem*{example*}{Example}		% for examples (unnumbered)

\newcounter{proofpart}

\DeclarePairedDelimiter{\abs}{\lvert}{\rvert}		% for absolute value
\DeclarePairedDelimiter{\floor}{\lfloor}{\rfloor}		% for floor
\DeclarePairedDelimiterX{\setdef}[2]{\{}{\}}{#1:#2}		% for set builder notation
\DeclarePairedDelimiterXPP{\exclude}[1]{\mathopen{}\setminus}{\{}{\}}{}{#1}

\newcommand{\eg}{e.g.,\xspace}		% for consistency
\newcommand{\ie}{i.e.,\xspace}		% for consistency
\newcommand{\textpar}[1]{\textup(#1\textup)}		% for upshape parentheses

\newcommand{\alt}[1]{#1'}		% for alternates
\newcommand{\N}{\mathbb{N}}		% for naturals
\newcommand{\R}{\mathbb{R}}		% for reals
\DeclareMathOperator{\bigoh}{\mathcal O}		% for Landau O
\newcommand{\dd}{\:d}		% for integrators
\newcommand{\eps}{\varepsilon}		% for better epsilon
\newcommand{\pd}{\partial}		% for derivatives
\newcommand{\vecspace}{\R^{\vdim}}		% for generic vector space
\newcommand{\vdim}{\debug d}		% for dimension
\DeclarePairedDelimiterX{\braket}[2]{\langle}{\rangle}{#1,#2}		% for duality pairing
\newcommand{\dspace}{\vecspace}		% for dual space
\newcommand{\dvec}{\debug y}		% for dual basis vectors
\newcommand{\mat}{G}		% for generic matrix

\DeclareMathOperator{\tcone}{TC}		% for tangent cone
\newcommand{\cvx}{\mathcal{C}}		% for generic convex set
\newcommand{\tvec}{\debug z}		% for tangent vectors
\DeclareMathOperator*{\argmin}{arg\,min}		% for argmin
\newcommand{\points}{\mathcal{\debug X}}		% for feasible set
\newcommand{\point}{\debug x}		% for primal variables
\newcommand{\pointalt}{\alt\point}		% for alternate primal variable

\newcommand{\dpoint}{\debug y}		% for generic dual variable
\newcommand{\obj}{\debug f}		% for objective function
\newcommand{\vecfield}{\debug V}		% for vector field
\newcommand{\vbound}{\debug M}		% for field bound

\newcommand{\sol}[1][\point]{#1^{\star}}		% for solutions (x by default)
\newcommand{\sols}{\sol[\points]}		% for set of solutions

\DeclareMathOperator{\Eucl}{\debug\Pi}		% for Euclidean projection
\newcommand{\radius}{\debug R}		% for Bregman divergence

\DeclareMathOperator{\ex}{\mathbb{E}}		% for expectations
\DeclareMathOperator{\prob}{\mathbb{P}}		% for probability
\newcommand{\samples}{\debug \Omega}		% for sample space
\newcommand{\filter}{\mathcal{\debug F}}		% for filtrations

\providecommand\given{}		% empty command for conditionals

\DeclarePairedDelimiterXPP{\exof}[1]{\ex}{[}{]}{}{%		% for conditional expectations
\renewcommand\given{\nonscript\:\delimsize\vert\nonscript\:\mathopen{}} #1}

\DeclarePairedDelimiterXPP{\probof}[1]{\prob}{(}{)}{}{%		% for conditional probabilities
\renewcommand\given{\nonscript\:\delimsize\vert\nonscript\:\mathopen{}} #1}

\newcommand{\noise}{\debug Z}		% for noise
\newcommand{\snoise}{\debug \xi}		% for scalar noise
\newcommand{\noisedev}{\debug \sigma}		% for noise stdev
\newcommand{\noisevar}{\debug \noisedev^{2}}		% for noise variance

\newcommand{\start}{\debug {1}}		% for start index
\newcommand{\running}{\debug 1,2,\dotsc}		% for running index
\newcommand{\run}{\debug t}		% for main sequence index
\newcommand{\runalt}{\debug s}		% for alternate sequence index
\newcommand{\nRuns}{\debug T}		% for total number of runs
\newcommand{\new}[1]{#1^{+}}		% for new iterate

\newcommand{\state}{\debug X}		% for primary iterate
\newcommand{\step}{\debug \gamma}		% for step-size
\DeclareMathOperator*{\intersect}{\cap}		% for intersections
\DeclareMathOperator*{\union}{\cup}		% for unions

\DeclareMathOperator{\one}{\mathds{1}}		% for indicator

\newcommand{\from}{\colon}		% for function definition
\DeclareMathOperator{\dist}{dist}		% for distance
\newcommand{\cpt}{\mathcal{\debug K}}						% for compacts
\newcommand{\nhd}{\debug U}		% for neighborhoods
\DeclarePairedDelimiterX{\product}[2]{\langle}{\rangle}{#1,#2}		% for scalar product
\DeclarePairedDelimiter{\norm}{\lVert}{\rVert}		% for norm
\newcommand{\dnorm}[1]{\norm{#1}}		% for dual norm

\DeclareMathOperator{\proj}{\Eucl}		% for projection
\newcommand{\ball}{\mathbb{\debug B}}		% for balls
\newcommand{\test}[1][\point]{\hat#1}		% for test points (x by default)
\newcommand{\arpoint}{\debug p}		% arbitrary point

\newcommand{\smalloh}{o}		% small o notation

\newcommand{\strong}{\debug \alpha}		% for strong convexity constant
\newcommand{\lips}{\debug \beta}       % for lipschitz constant
\newcommand{\scalar}{\debug \lambda}		% general scalar

\newcommand*{\defeq}{\coloneqq}      % for definition of new quantity
\newcommand*{\subs}{\leftarrow}      % for substitution

\newcommand{\ballr}[2]{\ball_{#2}(#1)}      % ball with radius and center

\newcommand{\Union}{\bigcup}
\newcommand{\disjUnion}{\dot{\Union}}        % disjoint union

\DeclareMathOperator{\Jac}{Jac}		% for Jacobian
\newcommand{\Jacf}[2]{\Jac_{#1}(#2)}    % Jacobian of function #1 evaluated at point #2

\newcommand{\nRunsNew}{\debug t}		% for historical reasons

\usepackage{pifont}% http://ctan.org/pkg/pifont
\DeclareMathSymbol{\qm}{\mathalpha}{operators}{"3F}

\newcommand{\sadobj}{\debug{\mathcal{L}}}		% for saddle point objective
\newcommand{\minvar}{\debug \theta}		% first variable of the saddle point objective
\newcommand{\maxvar}{\debug \phi}		% second variable of the saddle point objective
\newcommand{\minvars}{\debug \Theta}		% first variable space
\newcommand{\maxvars}{\debug \Phi}		% second variable space

\DeclareMathOperator{\nikiso}{\debug{NI}}		% nikaido-isoda
\DeclareMathOperator{\err}{Err}		% error function

\newcommand{\reserr}[1][\radius]{\err_{#1}}		% restricted error function
\newcommand{\resnikiso}[1][\radius]{\nikiso_{#1}}		% restricted nikaido-isoda

\newcommand{\laststart}{\debug 0}
\newcommand{\paststart}{\debug {1/2}}
\newcommand{\interstart}{\debug {3/2}}
\newcommand{\afterstart}{\debug 2}
\newcommand{\lastlaststart}{{\debug {-1}}}

\newcommand{\pointnew}{\new\point}		% for alternate primal variable
\newcommand{\intinterval}[2]{\{#1,...,#2\}}     % for integer interval between #1 and #2 (included)

\NewDocumentCommand{\current}{O{\state}O{\run}}{\debug{#1_{#2}}}
\NewDocumentCommand{\inter}{O{\state}O{\run}}{\debug{#1_{#2+1/2}}}
\NewDocumentCommand{\update}{O{\state}O{\run}}{\debug{#1_{#2+1}}}
\NewDocumentCommand{\last}{O{\state}O{\run}}{\debug{#1_{#2-1}}}
\NewDocumentCommand{\past}{O{\state}O{\run}}{\debug{#1_{#2-1/2}}}
\NewDocumentCommand{\pastpast}{O{\state}O{\run}}{\debug{#1_{#2-3/2}}}
\NewDocumentCommand{\future}{O{\state}O{\run}}{\debug{#1_{#2+3/2}}}
\NewDocumentCommand{\lastlast}{O{\state}O{\run}}{\debug{#1_{#2-2}}}

\newcommand{\avg}[1][\state]{\bar{#1}}

\newcommand{\seqinf}[3][\N]{(#2_#3)_{#3\in#1}}

\newcommand{\nhdalt}{\nhd_{\start}}		% neighborhood alter
\newcommand{\nhdradius}{\debug r}		% neighborhood radius
\newcommand{\smallproba}{\debug \delta}		% probability of going outside the neighborhood

\newcommand{\sumnoise}{\debug S}        % sum of first order terms
\newcommand{\sumnoisevar}{\debug R}     % sum of second order terms
\newcommand{\sumnoiseall}{\debug Q}     % S^2 + R
\newcommand{\noisebound}{\debug \eps}       % bound Q to define eventalt

\newcommand{\event}{\debug E}       % event
\newcommand{\eventalt}{\debug H}        % another event, changed from I to avoid conflict with identity matrix
\newcommand{\vecfieldstoch}{\vecfield}      % I will use it in the local convergence proof    
\newcommand{\compprob}[1]{{#1}^c}      % complementary for probability event
\newcommand{\setexclude}[2]{#1\setminus#2}
\newcommand{\lprbound}{\debug{\mathcal{M}}}       % use as bound in the local convergence proof
\newcommand{\stepsqsum}{\debug \Gamma}      % the sum of squared step-size

\newcommand{\cons}{\debug c}		% general constant
\newcommand{\consc}{\debug q}
\newcommand{\conscalt}{\alt\consc}
\newcommand{\seqitem}{\debug a}
\newcommand{\seqitemalt}{\alt \seqitem}
\newcommand{\intg}{\debug k}        % general integer (I avoid using n)
\newcommand{\tele}{\debug \mu}        % for telescopic general term
\newcommand{\strongstepm}{\debug b}		    % for strongly convex step-size denominator
\newcommand{\youngeps}{\debug \eps}        % Young's inequality with ε
\newcommand{\matfunc}{\debug \phi}

\newcommand{\matmin}{\debug A}
\newcommand{\matmax}{\debug B}
\newcommand{\matlin}{\debug C}
\newcommand{\indic}{\debug \epsilon}

\newcommand{\PM}[1]{\todo[color=DodgerBlue!30,author=\textbf{Pan:}]{\small #1\\}}
\newcommand{\YGH}[1]{\todo[color=Orchid!20!LightGray,author=\textbf{Yu-Guan:}]{\small #1\\}}

\begin{document}

%*************************************************************
%*****    FRONT MATTER AND METADATA
%*************************************************************

%----------------------------------------------------------------------
%%% TITLE & AUTHORS
%----------------------------------------------------------------------
\title{On the Convergence of Single-Call\\
Stochastic Extra-Gradient Methods}		% for title and runtitle

\author{%
Yu-Guan Hsieh\\
Univ. Grenoble Alpes, LJK and ENS Paris\\
38000 Grenoble, France.\\
\texttt{yu-guan.hsieh@ens.fr}
\And
Franck Iutzeler\\
Univ. Grenoble Alpes, LJK\\
38000 Grenoble, France.\\
\texttt{franck.iutzeler@univ-grenoble-alpes.fr}
\And
Jérôme Malick\\
CNRS, LJK\\
38000 Grenoble, France.\\
\texttt{jerome.malick@univ-grenoble-alpes.fr}
\And
Panayotis Mertikopoulos\\
Univ. Grenoble Alpes, CNRS, Inria, Grenoble INP, LIG\\
38000 Grenoble, France.\\
\texttt{panayotis.mertikopoulos@imag.fr}}

\maketitle

%----------------------------------------------------------------------
%%% KEYWORDS
%----------------------------------------------------------------------
%\subjclass[2010]{Primary XXXX, YYYY; Secondary ZZZZ.}
%\keywords{%
%separate;
%by;
%semicolon.}

%----------------------------------------------------------------------
%%% ACRONYMS
%----------------------------------------------------------------------
\newcommand{\acli}[1]{\textit{\acl{#1}}}		% for italicized acro
\newcommand{\aclip}[1]{\textit{\aclp{#1}}}		% for italicized acro (plural)
\newcommand{\acdef}[1]{\textit{\acl{#1}} \textup{(\acs{#1})}\acused{#1}}		% for acro def
\newcommand{\acdefp}[1]{\emph{\aclp{#1}} \textup(\acsp{#1}\textup)\acused{#1}}	% for acro def (plural)

\newacro{NI}{Nikaido\textendash Isoda}
\newacro{VI}{variational inequality}
\newacroplural{VI}[VIs]{variational inequalities}
\newacro{SVI}{Stampacchia variational inequality}
\newacro{MVI}{Minty variational inequality}
\newacro{SP}{saddle-point}

\newacro{GAN}{generative adversarial network}
\newacro{LHS}{left-hand side}
\newacro{RHS}{right-hand side}
\newacro{iid}[i.i.d.]{independent and identically distributed}
\newacro{lsc}[l.s.c.]{lower semi-continuous}
\newacro{NE}{Nash equilibrium}
\newacroplural{NE}[NE]{Nash equilibria}

\newacro{SGD}{stochastic gradient descent}
\newacro{FBF}{Forward-Backward-Forward}
\newacro{OMD}{Optimistic Mirror Descent}
\newacro{EG}{Extra-Gradient}
\newacro{MP}{Mirror-Prox}
\newacro{SEG}[$1$-EG]{single-call extra-gradient}
\newacro{PEG}{Past Extra-Gradient}
\newacro{RG}{Reflected Gradient}
\newacro{OG}{Optimistic Gradient}

%----------------------------------------------------------------------
%%% ABSTRACT
%----------------------------------------------------------------------
\begin{abstract}
\input{Abstract.tex}
\end{abstract}

%*************************************************************
%*****    BODY TEXT
%*************************************************************
\allowdisplaybreaks		% for breaking long displays
\acresetall		% for resetting acros

%----------------------------------------------------------------------
%%% INTRODUCTION
%----------------------------------------------------------------------
\section{Introduction}
\label{sec:intro}
\input{Introduction}

%----------------------------------------------------------------------
%%% SETUP
%----------------------------------------------------------------------
\section{Problem setup and blanket assumptions}
\label{sec:setup}
\input{Setup}

%----------------------------------------------------------------------
%%% ALGO
%----------------------------------------------------------------------
\vspace*{-1.5ex}
\section{Algorithms}
\label{sec:algorithms}
\input{Algos.tex}

%----------------------------------------------------------------------
%%% DETERMINISTIC
%----------------------------------------------------------------------
\section{Deterministic analysis}
\label{sec:det}
\input{Deterministic}

%----------------------------------------------------------------------
%%% STOCHASTIC
%----------------------------------------------------------------------
\section{Stochastic analysis}
\label{sec:stoch}
\input{Stochastic}

%----------------------------------------------------------------------
%%% EXPERIMENTS
%----------------------------------------------------------------------
%\section{Numerical experiments}
%\label{sec:experiments}
%\input{Experiments}

%----------------------------------------------------------------------
%%% CONCLUSIONS
%----------------------------------------------------------------------
\section{Concluding remarks}
\label{sec:conclusions}
\input{Conclusions}

%----------------------------------------------------------------------
%%% THANKS
%----------------------------------------------------------------------
\section*{Acknowledgments}
\input{Thanks}

%*************************************************************
%*****    BIBLIOGRAPHY
%*************************************************************
\bibliographystyle{ormsv080}
\bibliography{bibtex/IEEEabrv,bibtex/Bibliography-PM}

%*************************************************************
%*****    SUPPLEMENT
%*************************************************************
\newpage
\appendix
\numberwithin{equation}{section}		% for cleveref (when using cleveref and hyperref)
\numberwithin{lemma}{section}		% for cleveref (when using cleveref and hyperref)
\numberwithin{proposition}{section}		% for cleveref (when using cleveref and hyperref)
\numberwithin{theorem}{section}		% for cleveref (when using cleveref and hyperref)

%----------------------------------------------------------------------
%% Indexing
%----------------------------------------------------------------------
\renewcommand{\paststart}{\debug{\frac{1}{2}}}
\renewcommand{\interstart}{\debug{\frac{3}{2}}}

%----------------------------------------------------------------------
%% Algorithms
%----------------------------------------------------------------------
\RenewDocumentCommand{\inter}{O{\state}O{\run}}{\debug{#1_{#2+\frac{1}{2}}}}
\RenewDocumentCommand{\past}{O{\state}O{\run}}{\debug{#1_{#2-\frac{1}{2}}}}
\RenewDocumentCommand{\pastpast}{O{\state}O{\run}}{\debug{#1_{#2-\frac{3}{2}}}}
\RenewDocumentCommand{\future}{O{\state}O{\run}}{\debug{#1_{#2+\frac{3}{2}}}}

%%----------------------------------------------------------------------
%%%% APP: MERIT
%%----------------------------------------------------------------------
%\section{More on merit function}
%\label{app:merit}
%\input{App-Merit}

%%----------------------------------------------------------------------
%%%% APP: TECHNICAL
%%----------------------------------------------------------------------
\section{Technical lemmas}
\label{app:tech-lemmas}
\input{App-TechLemmas}

%%----------------------------------------------------------------------
%%%% APP: DETERMINISTIC
%%----------------------------------------------------------------------
\section{Proofs for the deterministic setting}
\label{app:det-proofs}
\input{App-DetProofs}

%%----------------------------------------------------------------------
%%%% APP: STOCHASTIC
%%----------------------------------------------------------------------
\section{Proofs for the stochastic setting}
\label{app:stoch-proofs}
\input{App-StochProofs}

\end{document}

%% file: Abstract.tex
%----------------------------------------------------------------------
%%% ABSTRACT
%----------------------------------------------------------------------
% !TEX root = ./Main.tex
%
%
Variational inequalities have recently attracted considerable interest in machine learning as a flexible paradigm for models that go beyond ordinary loss function minimization (such as generative adversarial networks and related deep learning systems).
In this setting, the optimal $\bigoh(1/\nRunsNew)$ convergence rate for solving smooth monotone \aclp{VI} is achieved by the \ac{EG} algorithm and its variants.
Aiming to alleviate the cost of an extra gradient step per iteration (which can become quite substantial in deep learning applications), several algorithms have been proposed as surrogates to \acl{EG} with a \emph{single} oracle call per iteration.
In this paper, we develop a synthetic view of such algorithms, and we complement the existing literature by showing that they retain a $\bigoh(1/\nRunsNew)$ ergodic convergence rate in smooth, deterministic problems.
Subsequently,
beyond the monotone deterministic case,
we also show that the last iterate of single-call, \emph{stochastic} extra-gradient methods still enjoys a $\bigoh(1/\nRunsNew)$ local convergence rate to solutions of \emph{non-monotone} \aclp{VI} that satisfy a second-order sufficient condition.

%% file: Introduction.tex
%----------------------------------------------------------------------
%%% INTRODUCTION
%----------------------------------------------------------------------
% !TEX root = ./Main.tex

Deep learning is arguably the fastest-growing field in artificial intelligence:
its applications range from image recognition and natural language processing to medical anomaly detection, drug discovery, and most fields where computers are required to make sense of massive amounts of data.
%As part of a broader family of machine learning methods that employ artificial neural networks, the use of continuous optimization techniques is a key component of this surging paradigm.
In turn, this has spearheaded a prolific research thrust in optimization theory with the twofold aim
%\begin{inparaenum}
%[\itshape a\upshape)]
%\item
of demystifying the successes of deep learning models
and
%\item
of providing novel methods to overcome their failures.
%\end{inparaenum}

%As the youngest poster child of this deep learning revolution, the \ac{GAN} framework intruduced by \citet{GPAM+14} has occupied the forefront of this drive in more ways than one.
Introduced by \citet{GPAM+14}, \acp{GAN} have become the youngest torchbearers of the deep learning revolution and have occupied the forefront of this drive in more ways than one.
First, the adversarial training of deep neural nets has given rise to new challenges regarding the efficient allocation of parallelizable resources, the compatibility of the chosen architectures, etc.
Second, the loss landscape in \acp{GAN} is no longer that of a minimization problem but that of a zero-sum, min-max game
\textendash\ or, more generally, a \acdef{VI}.

Variational inequalities are a flexible and widely studied framework in optimization which, among others, incorporates minimization, \acl{SP}, Nash equilibrium, and fixed point problems.
%, as well as a wide array of other problems where operators (or vector fields) are involved.%
%\footnote{In minimization problems, the operator is the subgradient of the loss function;
%in saddle-point problems, the operator is the associated Hamiltonian field.
%For a detailed presentation, see \cref{sec:setup}.}
As such, there is an extensive literature devoted to solving \aclp{VI} in different contexts;
for an introduction, see \cite{FP03,BC17} and references therein.
In particular, in the setting of monotone \aclp{VI} with Lipschitz continuous operators, it is well known that the optimal rate of convergence is $\bigoh(1/\nRunsNew)$, and that this rate is achieved by the \ac{EG} algorithm of \citet{Kor76} and its Bregman variant, the \ac{MP} algorithm of \citet{Nem04}.%
\footnote{\citet{Kor76} proved the method's asymptotic convergence for pseudomonotone \aclp{VI}.
The $\bigoh(1/\nRunsNew)$ convergence rate was later established by \citet{Nem04} with ergodic averaging.}
%\FI{``with ergodic average'' is maybe better?}
%see also \cite{Nes07} for a primal-dual variant known as ``dual extrapolation''.}

These algorithms require two projections and two oracle calls per iteration, so they are more costly than standard Forward-Backward / descent methods.
As a result, there are two complementary strands of literature aiming to reduce one (or both) of these cost multipliers \textendash\ that is,
%\ie
%\begin{inparaenum}
%[\itshape a\upshape)]
%\item
the number of projections
and/or
%\item
the number of oracle calls per iteration.
%\end{inparaenum}
The first class contains algorithms like the \ac{FBF} method of \citet{Tse00}, while the second focuses on gradient extrapolation mechanisms like Popov's modified Arrow\textendash Hurwicz algorithm \citep{Pop80}.

In deep learning, the latter direction has attracted considerably more interest than the former.
The main reason for this is that neural net training often does not involve constraints (and, when it does, they are relatively cheap to handle).
On the other hand, gradient calculations can become very costly, so a decrease in the number of oracle calls could offer significant practical benefits.
In view of this, our aim in this paper is
\begin{inparaenum}
[(\itshape i\upshape)]
\item
to develop a synthetic approach to methods that retain the anticipatory properties of the \acl{EG} algorithm while making a single oracle call per iteration;
and
\item
to derive quantitative convergence results for such \acdef{SEG} algorithms.
\end{inparaenum}

%----------------------------------------------------------------------
%% Overview table begins here

\begin{table}[tbp]
\centering
\footnotesize
\renewcommand{\arraystretch}{1.1}
\input{Tables/Overview}
\medskip
\caption{The best known global convergence rates for \acl{SEG} methods in monotone \ac{VI} problems;
logarithmic factors ignored throughout.
A box indicates a contribution from this paper.
%\revise{post-submission, we were also made aware of the very recent preprint \cite{MOP19b} which establishes an $\bigoh(1/\nRunsNew)$ deterministic ergodic convergence rate for smooth, monotone saddle-point problems.}
}
\label{tab:overview}
\vspace{-2ex}
\end{table}

%% Overview table ends here
%----------------------------------------------------------------------

%----------------------------------------------------------------------
%%% CONTRIBS
%----------------------------------------------------------------------
%\vspace{-1ex}
\Paragraph{Our contributions\afterhead}

Our first contribution complements the existing literature (reviewed below and in \cref{sec:algorithms}) by showing that the class of \ac{SEG} algorithms under study attains the optimal $\bigoh(1/\nRunsNew)$ convergence rate of the two-call method in deterministic \aclp{VI} with a monotone, Lipschitz continuous operator.
Subsequently, we show that this rate is also achieved in \emph{stochastic} \aclp{VI} with strongly monotone operators provided that the optimizer has access to an oracle with bounded variance (but not necessarily bounded second moments).

Importantly,
%these results are ``ergodic'', \ie they concern a weighted average of the sequence of points generated by the algorithm.
this stochastic result concerns both the method's
``ergodic average'' (a weighted average of the sequence of points generated by the algorithm)
as well as its
``last iterate'' (the last generated point).
The reason for this dual focus is that
averaging can be very useful in convex/monotone landscapes, but it is not as beneficial in non-monotone problems (where Jensen's inequality does not apply).
On that account, last-iterate convergence results comprise an essential stepping stone for venturing beyond monotone problems.
%Accordingly, as a first step beyond the ergodic regime, we also show that the algorithm's ``last iterate''
%%(\ie the last point at which the oracle was called)
%enjoys the same $\bigoh(1/\nRunsNew)$ convergence rate in stochastic, strongly monotone problems.

Armed with these encouraging results,
%given the highly convoluted loss landscape that arises in many applications of practical interest (such as \acp{GAN}),
we then focus on \emph{non-monotone} problems and show that, with high probability, the method's last iterate exhibits a $\bigoh(1/\nRunsNew)$ local convergence rate to solutions of non-monotone \aclp{VI} that satisfy a second-order sufficient condition.
To the best of our knowledge, this is the first convergence rate guarantee of this type for stochastic, non-monotone \aclp{VI}.%
%The fact that this rate coincides with the global $\bigoh(1/\nRunsNew)$ rate for deterministinc monotone \aclp{VI} suggests that the method is, in a certain sense, universal.

%----------------------------------------------------------------------
%%% RELATED
%----------------------------------------------------------------------
%\vspace{-1ex}
\Paragraph{Related work\afterhead}

The prominence of \acl{EG}/\acl{MP} methods in solving \aclp{VI} and \acl{SP} problems has given rise to a vast corpus of literature which we cannot hope to do justice here.
Especially in the context of adversarial networks, there has been a flurry of recent activity relating variants of the \acl{EG} algorithm to \ac{GAN} training, see \eg \cite{DISZ18,YSXJ+18,GBVV+19,GHPL+19,MLZF+19,CGFLJ19,LS19} and references therein.
For concreteness, we focus here on algorithms with a single-call structure and refer the reader to \cref{sec:algorithms,sec:det,sec:stoch} for additional details.

The first variant of \acl{EG} with a single oracle call per iteration dates back to \citet{Pop80}.
This algorithm was subsequently studied by, among others, \citet{CYLM+12}, \citet{RS13-NIPS,RS13-COLT} and \citet{GBVV+19};
see also \cite{Mal15,CS16} for a ``reflected'' variant, \cite{DISZ18,PDZC19,MOP19a,MOP19b} for an ``optimistic'' one, and \cref{sec:algorithms} for a discussion of the differences between these variants.
In the context of deterministic, strongly monotone \aclp{VI} with Lipschitz continuous operators, the last iterate of the method was shown to exhibit a geometric convergence rate \citep{Tse95,GBVV+19,Mal15,MOP19a};
similar geometric convergence results also extend to bilinear \acl{SP} problems \cite{Tse95,GBVV+19,PDZC19}, even though the operator involved is not strongly monotone.
In turn, this implies the convergence of the method's ergodic average, but at a $\bigoh(1/\nRunsNew)$ rate (because of the hysteresis of the average).
In view of this, the fact that \ac{SEG} methods retain the optimal $\bigoh(1/\nRunsNew)$ convergence rate in deterministic \aclp{VI} without strong monotonicity assumptions closes an important gap in the literature.%
\footnote{A few weeks after the submission of our paper, we were made aware of a very recent preprint by \citet{MOP19b} which also establishes a $\bigoh(1/\nRunsNew)$ convergence rate for the algorithm's ``optimistic'' variant in saddle-point problems (in terms of the \acl{NI} gap function).
To the best of our knowledge, this is the closest result to our own in the literature.}

At the local level, the geometric convergence results discussed above echo a surge of interest in local convergence guarantees of optimization algorithms applied to games and \acl{SP} problems,
see \eg \cite{LS19,pmlr-v89-adolphs19a,daskalakis2018limit,pmlr-v80-balduzzi18a} and references therein.
%In face of saddle point problems, the same second-order condition has been considered in \cite{LS19,pmlr-v89-adolphs19a,mazumdar2019finding}
%\citet{pmlr-v89-adolphs19a} proposed an algorithm exploiting curvature information for which all the stable stationary points are local nash equilibria.
%\YGH{In fact it is rather local strongly convex-concave problem. Do you think that we need to modify the sentence to make it a bit clearer?}
%\PM{Better now?}
In more detail, \citet{LS19} proved local geometric convergence for several algorithms in possibly non-monotone saddle-point problems under a local smoothness condition.
%Under weaker assumption, \citet{pmlr-v80-balduzzi18a} introduced an algorithm that is provably attracted to 
In a similar vein, \citet{daskalakis2018limit} analyzed the limit points of (optimistic) gradient descent, and showed that local saddle points are stable stationary points;
%Multiple works got interest in studying the stationary points/the local convergence property of several saddle-point algorithms.
subsequently, \citet{pmlr-v89-adolphs19a} and \citet{mazumdar2019finding} proposed a class of algorithms that eliminate stationary points which are not local \aclp{NE}.
%Among these works, the same second-order condition has been considered in \cite{LS19,pmlr-v89-adolphs19a,mazumdar2019finding}

Geometric convergence results of this type are inherently deterministic because they rely on an associated resolvent operator being firmly nonexpansive \textendash\ or, equivalently, rely on the use of the center manifold theorem.
In a stochastic setting, these techniques are no longer applicable because the contraction property cannot be maintained in the presence of noise;
in fact, unless the problem at hand is amenable to variance reduction \textendash\ \eg as in \cite{IJOT17,BMSV19,CGFLJ19} \textendash\  geometric convergence is not possible if the noise process is even weakly isotropic.
Instead, for monotone problems, \citet{CS16} and \citet{GBVV+19} showed that the ergodic average of the method attains a $\bigoh(1/\sqrt{\nRunsNew})$ convergence rate.
%(assuming a stochastic oracle with bounded second moments and bounded variance);
Our global convergence results for stochastic \aclp{VI} improve this rate to $\bigoh(1/\nRunsNew)$ in strongly monotone \aclp{VI} for both the method's ergodic average and its last iterate.
In the same light, our local $\bigoh(1/\nRunsNew)$ convergence results for \emph{non-monotone} \aclp{VI} provide a key extension of local, deterministic convergence results to a fully stochastic setting, all the while retaining the fastest convergence rate for monotone \aclp{VI}.

For convenience, our contributions relative to the state of the art are summarized in \cref{tab:overview}.

%% file: Tables/Overview.tex
%----------------------------------------------------------------------
%%% OVERVIEW
%----------------------------------------------------------------------
% !TEX root = ../Main.tex

\begin{tabular}{lcccccc}
\toprule
\multirow{2}{*}{}
	&\multicolumn{2}{c}{Lipschitz}
%	&\multicolumn{2}{c}{Strong}
	&\multicolumn{2}{c}{Lipschitz $+$ Strong}
	\\
	\cmidrule(lr){2-3}
	\cmidrule(lr){4-5}
%	\cmidrule(lr){6-7}
	&Ergodic
		&Last Iterate
%	&Ergodic
%		&Last Iterate
	&Ergodic
		&Last Iterate
	\\
	\midrule
	Deterministic
	&\Ovalbox{$1/\nRunsNew$}
		&Unknown
%	&\Ovalbox{$1/\nRunsNew$}
%		&\Ovalbox{$1/\nRunsNew$}
	&$1/\nRunsNew$
	&$e^{-\rho\nRunsNew}$ \cite{GBVV+19,MOP19a,Mal15} %,PDZC19}
	\\[\smallskipamount]
	Stochastic
	&$1/\sqrt{\nRunsNew}$ \cite{GBVV+19,CS16}
		&Unknown
%	&\Ovalbox{$1/\nRunsNew$}
%		&\Ovalbox{$1/\nRunsNew$}
	&\Ovalbox{$1/\nRunsNew$}
		&\Ovalbox{$1/\nRunsNew$}
	\\[.5ex]
	\bottomrule
\end{tabular}

%% file: Setup.tex
%----------------------------------------------------------------------
%%% SETUP
%----------------------------------------------------------------------
% !TEX root = ./Main.tex

\Paragraph{Variational inequalities\afterhead}
We begin by presenting the basic \acl{VI} framework that we will consider throughout the sequel.
To that end,
let $\points$ be a nonempty closed convex subset of $\R^{\vdim}$,
and let $\vecfield\from\dspace\to\dspace$ be a single-valued operator on $\vecspace$.
In its most general form, the \acdef{VI} problem associated to $\vecfield$ and $\points$ can be stated as:
\begin{equation}
\label{eq:SVI}
\tag{VI}
\text{Find $\sol\in\points$ such that $\braket{\vecfield(\sol)}{\point - \sol} \geq 0$ for all $\point\in\points$}.
\end{equation}
To provide some intuition about \eqref{eq:SVI},
%This formulation is often referred to as a \acdef{SVI};
%by contrast, the \acdef{MVI} associated to $\vecfield$ is
%\begin{equation}
%\label{eq:MVI}
%\tag{MVI}
%\text{Find $\sol\in\points$ such that $\braket{\vecfield(\point)}{\point - \sol} \geq 0$ for all $\point\in\points$}.
%\end{equation}
%To provide intuition about these two formulations,%
%\footnote{Some authors refer to solutions of \ac{SVI} and \ac{MVI} as ``strong'' and ``weak'' respectively \citep{Nes09,JNT11};
%we do not adopt this terminology to avoid clashing with the notion of a ``strong'' solution for multi-valued operators \citep{DH99}.}
we discuss two important examples below:
\medskip

\begin{example}[Loss minimization]
\label{ex:function}
Suppose that $\vecfield = \nabla\obj$ for some smooth loss function $\obj$ on $\points = \R^{\vdim}$.
Then, $\sol\in\points$ is a solution to \eqref{eq:SVI} if and only if $\nabla\obj(\sol) = 0$, \ie if and only if $\sol$ is a critical point of $\obj$.
Of course, if $\obj$ is convex, any such solution is a global minimizer.\qed
%By comparison, solutions to \eqref{eq:MVI} are global minimizers of $\obj$, but the converse does not hold (\eg if $\obj(\point) = \point^{4}/4 - \point^{2}/2$ and $\sol = \pm1$).
\end{example}
\smallskip

\begin{example}[Min-max optimization]
\label{ex:saddle}
Suppose that $\points$ decomposes as $\points = \minvars\times\maxvars$ with $\minvars = \R^{\vdim_{1}}$, $\maxvars = \R^{\vdim_{2}}$, and assume $\vecfield = (\nabla_{\minvar}\sadobj,-\nabla_{\maxvar}\sadobj)$ for some smooth function $\sadobj(\minvar,\maxvar)$, $\minvar\in\minvars$, $\maxvar\in\maxvars$.
As in \cref{ex:function} above, the solutions to \eqref{eq:SVI} correspond to the critical points of $\sadobj$;
if, in addition, $\sadobj$ is convex-concave, any solution $\sol = (\sol[\minvar],\sol[\maxvar])$ of \eqref{eq:SVI} is a global \acli{SP}, \ie
%now however, if $\sol = (\sol[\minvar],\sol[\maxvar])$ solves \eqref{eq:MVI}, it is a global \acli{SP} of $\sadobj$, \ie
%\begin{equation}
%\sadobj(\sol[\minvar],\sol[\maxvar])
%	\leq \sadobj(\minvar,\sol[\maxvar])
%	\quad
%	\text{and}
%	\quad
%\sadobj(\sol[\minvar],\sol[\maxvar])
%	\geq \sadobj(\sol[\minvar],\maxvar)
%%	\quad
%%	\text{for all $\minvar\in\minvars$, $\maxvar\in\maxvars$}.
%\end{equation}
%for all $\minvar\in\minvars$ and all $\maxvar\in\maxvars$.
\begin{equation}
\sadobj(\sol[\minvar],\maxvar) \leq \sadobj(\sol[\minvar],\sol[\maxvar])
	\leq \sadobj(\minvar,\sol[\maxvar])
	\quad
	\text{for all $\minvar\in\minvars$ and all $\maxvar\in\maxvars$.}
%	\quad
%	\text{for all $\minvar\in\minvars$, $\maxvar\in\maxvars$}.
\end{equation}
Given the original formulation of \acp{GAN} as (stochastic) \acl{SP} problems \cite{GPAM+14}, this observation has been at the core of a vigorous literature at the interface between optimization, game theory, and deep learning, see \eg \cite{DISZ18,YSXJ+18,MLZF+19,GBVV+19,PDZC19,LS19,CGFLJ19} and references therein.\qed
\end{example}
\smallskip

The operator analogue of convexity for a function is \emph{monotonicity}, \ie
%An important class of operators in which the formulations \eqref{eq:SVI} and \eqref{eq:MVI} coincide is when $\vecfield$ is \emph{monotone}, \ie
\begin{equation}
\label{eq:mono}
%\tag{Mon}
\braket{\vecfield(\pointalt) - \vecfield(\point)}{\pointalt - \point}
	\geq 0
	\quad
	\text{for all $\point,\pointalt \in \vecspace$}.
\end{equation}
Specifically, when $\vecfield = \nabla\obj$ for some sufficiently smooth function $\obj$, this condition is equivalent to $\obj$ being convex \citep{BC17}.
%As such, \eqref{eq:mono} can be seen as an analogue of convexity for non-gradient operators, and the equivalence between \eqref{eq:SVI} and \eqref{eq:MVI} in this case is formally analogous to the equivalence between critical points and global minimizers for convex functions.
%\YGH{How about convex-concave saddle point problems? Do we mention it?}
%\PM{I don't think we need to, but if you want to include a short phrase, sure, go for it, why not}
%In a similar vein, the same analogy extends to \emph{strongly monotone} operators, \ie operators satisfying the property
%\begin{equation}
%\label{eq:mono-strong}
%\braket{\vecfield(\pointalt) - \vecfield(\point)}{\pointalt - \point}
%	\geq \strong \norm{\pointalt - \point}^{2}
%	\quad
%	\text{for some $\strong>0$ and all $\point,\pointalt \in \points$}.
%\end{equation}
%Specifically, $\vecfield=\nabla\obj$ is $\strong$-strongly monotone if and only if $\obj$ is $\strong$-strongly convex.
%
%If $\vecfield$ admits a potential, the quality of a candidate solution $\test\in\points$ can be assessed via the value difference $\obj(\test) - \min\obj$.
In this case, following \citet{Nes07,Nes09} and \citet{JNT11}, the quality of a candidate solution $\test\in\points$ can be assessed via the so-called \emph{error} (or \emph{merit}) \emph{function}
\begin{flalign}
\label{eq:error}
\err(\test)
	&= \sup_{\point\in\points} \braket{\vecfield(\point)}{\test - \point}
\intertext{and/or its restricted variant}
\label{eq:error-res}
\reserr(\test)
	&= \max_{\point\in\points_{\radius}} \braket{\vecfield(\point)}{\test - \point},
\end{flalign}
where $\points_{\radius} \equiv \points\cap\ballr{0}{\radius} = \setdef{\point\in\points}{\norm{\point} \leq \radius}$ denotes the ``restricted domain'' of the problem.
%\PM{Did not put a reference point here \textendash\ I think it would be simpler to say later that the algorithms are initialized at the minimum-norm element of $\points$, but we can change if you prefer. Let's play this by ear.}
More precisely, we have the following basic result.

\begin{lemma}[\citeauthor{Nes07}, \citeyear{Nes07}]
\label{lem:error}
Assume $\vecfield$ is monotone.
If $\sol$ is a solution of \eqref{eq:SVI}, we have $\err(\sol) = 0$ and $\reserr(\sol) = 0$ for all sufficiently large $\radius$.
Conversely, if $\reserr(\test)=0$ for large enough $\radius>0$ and some $\test\in\points_{\radius}$, then $\test$ is a solution of \eqref{eq:SVI}.
\end{lemma}

\YGH{We use mean squared error in the strongly monotone case.}
In light of this result,
$\err$ and $\reserr$ will be among our principal measures of convergence in the sequel.

\Paragraph{Blanket assumptions\afterhead}
%In this above setting, we will develop our convergence under standard assumptions on $V$. More precisely, our blanket assumptions in this paper are the following.
With all this in hand, we present below the main assumptions that will underlie the bulk of the analysis to follow.

\begin{assumption}
\label{asm:solution}
The solution set $\sols$ of \eqref{eq:SVI} is nonempty.
\end{assumption}

\begin{assumption}
\label{asm:Lipschitz}
The operator $\vecfield$ is $\lips$-Lipschitz continuous, \ie
\begin{equation}
\label{eq:Lipschitz}
\dnorm{\vecfield(\pointalt) - \vecfield(\point)}
	\leq \lips \norm{\pointalt - \point}
	\quad
	\text{for all $\point,\pointalt\in\vecspace$}.
\end{equation}
\end{assumption}

\begin{assumption}
\label{asm:mono}
The operator $\vecfield$ is monotone.
\end{assumption}

%\vspace*{-1ex}

%In the above, $\dnorm{\cdot}$ denotes the dual norm of $\norm{\cdot}$, \ie $\dnorm{\dvec} = \max\setdef{\abs{\braket{\dvec}{\tvec}}}{\tvec\in\vecspace,\norm{\tvec}\leq 1}$.
In some cases, we will also strengthen \cref{asm:mono} to:

\asmtag{\ref*{asm:mono}\textpar{s}}
\begin{assumption}
\label{asm:mono-strong}
The operator $\vecfield$ is $\strong$-strongly monotone, \ie
\begin{equation}
\label{eq:mono-strong}
\braket{\vecfield(\pointalt) - \vecfield(\point)}{\pointalt - \point}
	\geq \strong \norm{\pointalt - \point}^{2}
	\quad
	\text{for some $\strong>0$ and all $\point,\pointalt \in \vecspace$}.
\end{equation}
\end{assumption}

Throughout our paper, we will be interested in sequences of points $\state_{\run} \in \points$ generated by algorithms that can access the operator $\vecfield$ via a \emph{stochastic oracle} \citep{Nes04}.%
\footnote{Depending on the algorithm, the sequence index $\run$ may take positive integer or half-integer values (or both).}
Formally, this is a black-box mechanism which, when called at $\current\in\points$, returns the estimate
\begin{equation}
\label{eq:oracle}
\current[\vecfield]
	= \vecfield(\current) + \current[\noise],
\end{equation}
where $\current[\noise]\in\dspace$ is an additive noise variable satisfying the following hypotheses:
\begin{subequations}
\label{eq:noise}
\begin{alignat}{2}
&a)\;\;
	\textit{Zero-mean:}
	&\quad
	&\exof{\current[\noise] \given \current[\filter]}
		= 0.
		\hspace{17em}
	\label{eq:mean}
	\\
&b)\;\;
	\textit{Finite variance:}
	&\quad
	&\exof{\dnorm{\current[\noise]}^{2} \given \current[\filter]}
		\leq \noisevar.
	\label{eq:variance}
\end{alignat}
\end{subequations}
In the above, $\current[\filter]$ denotes the history (natural filtration) of $\current[\state]$,
so $\state_{\run}$ is adapted to $\current[\filter]$ by definition;
on the other hand, since the $\run$-th instance of $\current[\noise]$ is generated randomly from $\current[\state]$, $\current[\noise]$ is \emph{not} adapted to $\current[\filter]$.
Obviously, if $\noisevar=0$, we have the deterministic, \emph{perfect feedback} case $\current[\vecfield] = \vecfield(\current)$.

%% file: Algos.tex
%----------------------------------------------------------------------
%%% ALGORITHMS
%----------------------------------------------------------------------
% !TEX root = ./Main.tex

\Paragraph{The \acl{EG} algorithm\afterhead}
In the general framework outlined in the previous section, the \acf{EG} algorithm of \citet{Kor76} can be stated in recursive form as
\begin{equation}
\label{eq:EG}
\tag{EG}
\begin{aligned}
\inter
	&= \proj_{\points}(\current - \current[\step] \current[\vecfield])
	\\
\update
	&= \proj_{\points}(\current - \current[\step] \inter[\vecfield])
\end{aligned}
\end{equation}
where
$\proj_{\points}(\dvec) \defeq \argmin_{\point\in\points} \norm{\dvec - \point}$ 
denotes the Euclidean projection of $\dvec\in\dspace$ onto the closed convex set $\points$
and
$\current[\step] > 0$ is a variable step-size sequence.
Using this formulation as a starting point, the main idea behind the method can be described as follows:
at each $\run=\running$, the oracle is called at the algorithm's current 
\textendash\ or \emph{base} \textendash\ 
state $\current$ to generate an 
intermediate \textendash\ or \emph{leading} \textendash\
state $\inter$;
subsequently, the base state $\current$ is updated to $\update$ using gradient information from the leading state $\inter$, and the process repeats.
Heuristically, the extra oracle call allows the algorithm to ``anticipate'' the landscape of $\vecfield$ and, in so doing, to achieve improved convergence results relative to standard projected gradient / forward-backward methods;
for a detailed discussion, we refer the reader to \cite{FP03,Bub15} and references therein.

%%----------------------------------------------------------------------
%%% Meta algorithm begins here
%
%\begin{algorithm}[t]
%\caption{\acf{SEG} meta-algorithm}
%\label{alg:MG}
%\input{Algorithms/MG}
%\end{algorithm}
%
%%% Meta algorithm ends here
%%----------------------------------------------------------------------

\Paragraph{Single-call variants of the \acl{EG} algorithm\afterhead}

Given the significant computational overhead of gradient calculations, a key desideratum is to drop the second oracle call in \eqref{eq:EG} while retaining the algorithm's ``anticipatory'' properties.
In light of this, we will focus on methods that perform a \emph{single} oracle call at the leading state $\inter$, but replace the update rule for $\inter$ (and, possibly, $\current$ as well) with a proxy that compensates for the missing gradient.
Concretely, we will examine the following family of \acdef{SEG} algorithms:

\begin{enumerate}[leftmargin=2em]
\addtolength{\itemsep}{\smallskipamount}

\item
\acdef{PEG} \cite{Pop80,CYLM+12,GBVV+19}:
\begin{alignat}{1}
\label{eq:PEG}
\tag{PEG}
\begin{split}
\inter
	&= \proj_{\points}(\current - \current[\step]\past[\vecfield])
	\\
\update
	&= \proj_{\points}(\current - \current[\step]\inter[\vecfield])
\end{split}
\intertext{
[Proxy:
use $\past[\vecfield]$ instead of $\current[\vecfield]$ in the calculation of $\inter$]}
\intertext{
\item
\acdef{RG} \cite{ChaPoc11,Mal15,CS16}:}
\label{eq:RG}
\tag{RG}
\begin{split}
\inter
	&= \current - (\last - \current)
	\\
\update
	&= \proj_{\points}(\current - \current[\step] \inter[\vecfield])
\end{split}
\intertext{[Proxy:
use $(\last - \current)/\current[\step]$ instead of $\current[\vecfield]$ in the calculation of $\inter$;
no projection]}%
\intertext{
\item
\acdef{OG} \cite{DISZ18,MOP19a,MOP19b,PDZC19}:}
\label{eq:OG}
\tag{OG}
\begin{split}
\inter
	&= \proj_{\points}(\current - \current[\step] \past[\vecfield])
	\\
\update
	&= \inter + \current[\step] \past[\vecfield] - \current[\step] \inter[\vecfield]
\end{split}
\end{alignat}
[Proxy:
use $\past[\vecfield]$ instead of $\current[\vecfield]$ in the calculation of $\inter$;
use $\inter + \current[\step] \past[\vecfield]$ instead of $\current$ in the calculation of $\update$;
no projection]
\end{enumerate}
%\vspace{-\smallskipamount}

These are the main algorithmic schemes that we will consider, so a few remarks are in order.
%\PM{Cite centripetal paper in OG?}
First, given the extensive literature on the subject, this list is not exhaustive;
see \eg
\cite{MOP19a,MOP19b,PDZC19} for a generalization of \eqref{eq:OG},
\cite{Mal19} for a variant that employs averaging to update the algorithm's base state $\current$,
and
\cite{GHPL+19} for a proxy defined via ``negative momentum''.
%\YGH{Minor notes: I notice that we use inconsistently \ac{OG} or \eqref{eq:OG}. I am not sure if there is a rule for this. To be fixed at the end.}
%\PM{I refer to \eqref{eq:OG} when I want to specify the actual equation, but I use \ac{OG} when I want to say the \ac{OG} algorithm. Sometimes I mess up\dots}
Nevertheless, the algorithms presented above appear to be the most widely used single-call variants of \eqref{eq:EG}, and they illustrate very clearly the two principal mechanisms for approximating missing gradients:
\begin{inparaenum}
[\upshape(\itshape i\hspace*{.5pt}\upshape)]
\item
using past gradients (as in the \ac{PEG} and \ac{OG} variants);
and/or
\item
using a difference of successive states (as in the \ac{RG} variant).
\end{inparaenum}

We also take this opportunity to provide some background and clear up some issues on terminology regarding the methods presented above.
First, the idea of using past gradients dates back at least to \citet{Pop80}, who introduced 
\eqref{eq:PEG} as a ``modified Arrow\textendash Hurwicz'' method a few years after the original paper of \citet{Kor76};
the same algorithm is called ``meta'' in \cite{CYLM+12} and ``extrapolation from the past'' in \cite{GBVV+19} (but see also the note regarding optimism below).
The terminology ``\acl{RG}'' and the precise formulation that we use here for \eqref{eq:RG} is due to \citet{Mal15}.
The well-known primal-dual algorithm of \citet{ChaPoc11} can be seen as a one-sided, alternating variant of the method for \acl{SP} problems;
see also \cite{YSXJ+18} for a more recent take.

Finally, the terminology ``optimistic'' is due to \citet{RS13-COLT,RS13-NIPS}, who provided a unified view of \eqref{eq:PEG} and \eqref{eq:EG} based on the sequence of oracle vectors used to update the algorithm's leading state $\inter$.% 
\footnote{More precisely, \citet{RS13-COLT,RS13-NIPS} use the term \acf{OMD} in reference to the \acl{MP} method of \citet{Nem04}, itself a variant of \eqref{eq:EG} with projections defined by means of a Bregman function;
for a related treatment, see \citet{Nes07} and \citet{JNT11}.}
Because the framework of \cite{RS13-COLT,RS13-NIPS} encompasses two different algorithms, there is some danger of confusion regarding the use of the term ``optimism'';
in particular, both \eqref{eq:EG} and \eqref{eq:PEG} can be seen as instances of optimism.
The specific formulation of \eqref{eq:OG} that we present here is the projected version of the algorithm considered by \citet{DISZ18};%
\footnote{To see this, note that the difference between two consecutive intermediate steps $\past$ and $\inter$ can be written as $\inter = \proj_{\points}(\past - (\last[\step]+\current[\step])\past[\vecfield] + \last[\step]\pastpast[\vecfield])$.
Writing \eqref{eq:OG} in the form presented above shows that \eqref{eq:OG} can also be viewed as a single-call variant of the \ac{FBF} method of \citet{Tse00}.}
%\footnote{Depending on when the projection is taken, \eqref{eq:OG} can be written instead as $\update = \proj_{\points}(\current - 2\current[\step]\current[\vecfield] + \current[\step]\last[\vecfield])$.
%We focus here on the specific presentation \eqref{eq:OG} because it can be seen as a single-call variant of the \ac{FBF} method of \citet{Tse00}.}
by contrast, the ``optimistic'' method of \citet{MLZF+19} is equivalent to \eqref{eq:EG} \textendash\ not \eqref{eq:PEG} or \eqref{eq:OG}.

The above shows that there can be a broad array of \aclp{SEG} methods depending on the specific proxy used to estimate the missing gradient, whether it is applied to the algorithm's base or leading state, when (or where) a projection operator is applied, etc.
The contact point of all these algorithms is the unconstrained setting ($\points = \vecspace$) where they are exactly equivalent:
\smallskip

\begin{proposition}
\label{prop:variants}
Suppose that the \ac{SEG} methods presented above share the same initialization, $\state_{\laststart} = \state_{\start} \in \points$, $\vecfield_{\paststart} = 0$, and are run with the same, constant step-size $\current[\step] \equiv \step$ for all $\run\geq1$.
If $\points = \vecspace$, the generated iterates $\current$ coincide for all $\run\geq1$.
\end{proposition}

The proof of this proposition follows by a simple rearrangement of the update rules for \eqref{eq:PEG}, \eqref{eq:RG} and \eqref{eq:OG}, so we omit it.
In the projected case, the \ac{SEG} updates presented above are no longer equivalent \textendash\ though, of course, they remain closely related.
%\footnote{For instance, one can consider the simple case of $\points = [0,1]^{\vdim}$}
% \FI{On donne un exemple de fonction aussi?}
% \PM{A figure would be better, but there's no space, and it's not important anyway.
% I took out the offending footnote and changed the phrase.}

%% file: Deterministic.tex
%----------------------------------------------------------------------
%%% DETERMINISTIC
%----------------------------------------------------------------------
% !TEX root = ./Main.tex

We begin with the deterministic analysis, \ie when the optimizer receives oracle feedback of the form \eqref{eq:oracle} with $\noisedev=0$.
In terms of presentation, we keep the global and local cases separated and we interleave our results for the generated sequence $\current$ and its
\emph{ergodic average}.
To streamline our presentation, we defer the details of the proofs to the paper's supplement and only discuss here the main ideas.

%----------------------------------------------------------------------
%%% GLOBAL
%----------------------------------------------------------------------
\subsection{Global convergence}
\label{sec:det-global}

Our first result below shows that the algorithms under study achieve the optimal $\bigoh(1/\nRunsNew)$ ergodic convergence rate in monotone problems with Lipschitz continuous operators.

\smallskip
\begin{theorem}
\label{thm:det-global-erg}
Suppose that $\vecfield$ satisfies \cref{asm:solution,asm:Lipschitz,asm:mono}.
Assume further that a \ac{SEG} algorithm is run with perfect oracle feedback and a constant step-size $\step < 1/(\cons\lips)$, where $\cons = 1 + \sqrt{2}$ for the \ac{RG} variant and $\cons = 2$ for the \ac{PEG} and \ac{OG} variants.
Then, for %sufficiently large
all $\radius>0$, we have
\begin{equation}
\label{eq:rate-det-erg}
\reserr\left(\avg_{\nRunsNew}\right)
	\leq \frac{\radius^{2}+\norm{\state_\start-\state_{\paststart}}^2}{2\step\nRunsNew}
\end{equation}
where  $\avg_{\nRunsNew} = \nRunsNew^{-1} \sum_{\runalt=\start}^{\nRunsNew} \inter[\state][\runalt]$ is the ergodic average of the algorithm's sequence of leading states.
\end{theorem}

This result shows that
%at least in the monotone case,
the \ac{EG} and \ac{SEG} algorithms share the same convergence rate guarantees, so we can safely drop one gradient calculation per iteration in the monotone case.
The proof of the theorem is based on the following technical lemma which enables us to treat the different variants of the \ac{SEG} method in a unified way.

\begin{lemma}
\label{lem:descent}
Assume that $\vecfield$ satisfies \cref{asm:mono} \textpar{monotonicity}.
Suppose further that the sequence $(\current)_{\run\in\N/2}$ of points in $\vecspace$ satisfies the following ``quasi-descent'' inequality with 
$\current[\tele][\runalt],\current[\scalar][\runalt]\geq0$:
\begin{equation}
\label{eq:descent}
\norm{\update[\state][\runalt] - \arpoint}^{2}
	\leq \norm{\current[\state][\runalt] - \arpoint}^{2}
	- 2 \current[\scalar][\runalt] \product{\vecfield(\inter[\state][\runalt])}{\inter[\state][\runalt] - \arpoint}
	+ \current[\tele][\runalt] - \update[\tele][\runalt]
\end{equation}
for all $\arpoint\in\points_\radius$
and all $\runalt\in\{\start,\dotsc,\nRunsNew\}$.
Then, 
\begin{equation}
\label{eq:template}
\reserr\left(
    \frac{\sum_{\runalt=\start}^{\nRunsNew} \current[\scalar][\runalt] \inter[\state][\runalt]}{\sum_{\runalt=\start}^{\nRunsNew} \current[\scalar][\runalt]}\right)
	\leq \frac{\radius^{2} + \tele_\start}{2 \sum_{\runalt=\start}^{\nRunsNew} \current[\scalar][\runalt]}.
\end{equation}
% where  $\avg_{\nRunsNew}^{\scalar} = \sum_{\runalt=\start}^{\nRunsNew} \current[\scalar][\runalt] \inter[\state][\runalt] \big/ \sum_{\runalt=\start}^{\nRunsNew} \current[\scalar]$ is the ergodic average associated with $(\current[\scalar][\runalt])_{\runalt\in\intinterval{\start}{\nRunsNew}}$.
\end{lemma}

\vspace{0.4em}
\begin{remark}
For \cref{ex:function,ex:saddle} it is possible to state both \cref{thm:det-global-erg} and \cref{lem:descent} with more adapted measures.
We refer the readers to the supplement for more details.
\end{remark}

The use of \cref{lem:descent} is tailored to time-averaged sequences like $\avg_{\nRunsNew}$, and relies on establishing a suitable ``quasi-descent inequality'' of the form \eqref{eq:descent} for the iterates of \ac{SEG}.
Doing this requires in turn a careful comparison of successive iterates of the algorithm via the Lipschitz continuity assumption for $\vecfield$;
we defer the precise treatment of this argument to the paper's supplement.

On the other hand, because the role of averaging is essential in this argument, the convergence of the algorithm's last iterate requires significantly different techniques.
To the best of our knowledge, there are no comparable convergence rate guarantees for $\current$ under \cref{asm:solution,asm:Lipschitz,asm:mono};
however, if \cref{asm:mono} is strengthened to \cref{asm:mono-strong}, the convergence of $\current$ to the (necessarily unique) solution of \eqref{eq:SVI} occurs at a geometric rate.
For completeness, we state here a consolidated version of the geometric convergence results of \citet{Mal15}, \citet{GBVV+19}, and \citet{MOP19a}.

\begin{theorem}
\label{thm:det-global-last}
Assume that $\vecfield$ satisfies \cref{asm:solution,asm:Lipschitz,asm:mono-strong}, and let $\sol$ denote the \textpar{necessarily unique} solution of \eqref{eq:SVI}.
If a \ac{SEG} algorithm is run with a sufficiently small step-size $\step$, the generated sequence $\current$ converges to $\sol$ at a rate of $\norm{\current - \sol} = \bigoh(\exp(-\rho\,\run))$ for some $\rho>0$.
\end{theorem}

%----------------------------------------------------------------------
%%% LOCAL
%----------------------------------------------------------------------
\subsection{Local convergence}
\label{sec:det-local}

We continue by presenting a local convergence result for deterministic, \emph{non-monotone} problems.
To state it, we will employ the following notion of regularity in lieu of \cref{asm:solution,asm:Lipschitz,asm:mono} and \labelcref{asm:mono-strong}.

\begin{definition}
\label{def:regular}
We say that $\sol$ is a \emph{regular solution} of \eqref{eq:SVI} if $\vecfield$ is $C^{1}$-smooth in a neighborhood of $\sol$ and the Jacobian $\Jacf{\vecfield}{\sol}$ is positive-definite along rays emanating from $\sol$, \ie
%\FI{Added at $\sol$}
%\PM{Ok}
\begin{equation}
\label{eq:Jac}
\tvec^{\top} \Jacf{\vecfield}{\sol} \tvec
	\equiv \sum_{i,j=1}^{\vdim} \tvec_{i} \frac{\pd\vecfield_{i}}{\pd\point_{j}}(\sol) \tvec_{j}
	> 0
	\quad
%	\text{for all $\tvec\in\vecspace\setminus\!\{0\}$ that are tangent to $\points$ at $\sol$}. 
\end{equation}
for all $\tvec\in\vecspace\setminus\!\{0\}$ that are tangent to $\points$ at $\sol$.
\end{definition}

% \begin{remark}
% \label{rem:RG-regular}
% \revise{Because $\ac{RG}$ is an infeasible method, the tangent cone condition in \cref{def:regular} needs to be extended to the tangent space to $\points$ at $\point$;
% for simplicity, we will not make this distinction in the sequel.}
% \PM{I changed the text as it would be silly to initialize \ac{RG} outside $\points$.}
% \end{remark}

This notion of regularity  is an extension of similar conditions that have been employed in the local analysis of loss minimization and \acl{SP} problems.
More precisely, if $\vecfield = \nabla\obj$ for some loss function $\obj$, this definition is equivalent to positive-definiteness of the Hessian along qualified constraints \citep[Chap.\,3.2]{bertsekas1997nonlinear}.
%\PM{Jérôme, leaving this to you.} 
As for \acl{SP} problems and smooth games, variants of this condition can be found in several different sources, see \eg \cite{Ros65,FK07,MZ19,ratliff2013characterization,LS19} and references therein.

Under this condition, we obtain the following local geometric convergence result for \ac{SEG} methods.

\medskip
\begin{theorem}
\label{thm:det-local}
Let $\sol$ be a regular solution of \eqref{eq:SVI}.
If a \ac{SEG} method is run with perfect oracle feedback and is initialized sufficiently close to $\sol$ with a sufficiently small constant step-size,%
%\footnote{\revise{
%Whenever the statement of the theorem involves expressions such as ``sufficiently small'' or ``large enough'', this means that the exact bound depends on some local constant making the explicit expression hard to be written here.
%The interested readers may refer to the proof of the respective theorem to get deeper insight on how these quantities should be chosen.}}
we have $\norm{\current - \sol} = \bigoh(\exp(-\rho\,\run))$ for some $\rho>0$.
\end{theorem}

The proof of this theorem relies on showing that
\begin{inparaenum}
[(\itshape i\hspace*{.5pt}\upshape)]
\item
$\vecfield$ essentially behaves like a smooth, strongly monotone operator close to $\sol$;
and
\item
if the method is initialized in a small enough neighborhood of $\sol$, it will remain in said neighborhood for all $\run$.
\end{inparaenum}
As a result, \cref{thm:det-local} essentially follows by ``localizing'' \cref{thm:det-global-last} to this neighborhood.

As a preamble to our stochastic analysis in the next section, we should state here that, albeit straightforward, the proof strategy outlined above breaks down if we have access to $\vecfield$ only via a \emph{stochastic} oracle.
In this case, a single ``bad'' realization of the feedback noise $\current[\noise]$ could drive the process away from the attraction region of any local solution of \eqref{eq:SVI}.
For this reason, the stochastic analysis requires significantly different tools and techniques and is considerably more intricate.

%% file: Stochastic.tex
%----------------------------------------------------------------------
%%% STOCHASTIC
%----------------------------------------------------------------------
% !TEX root = ./Main.tex

We now present our analysis for stochastic \aclp{VI} with oracle feedback of the form \eqref{eq:oracle}.
For concreteness, given that the \ac{PEG} variant of the \ac{SEG} method employs the most straightforward proxy mechanism, we will focus on this variant throughout;
for the other variants, the proofs and corresponding explicit expressions follow from the same rationale (as in the case of \cref{thm:det-global-erg}).

%----------------------------------------------------------------------
%%% GLOBAL
%----------------------------------------------------------------------
\subsection{Global convergence}
\label{sec:stoch-global}

As we mentioned in the introduction, under \cref{asm:solution,asm:Lipschitz,asm:mono}, \citet{CS16} and \citet{GBVV+19} showed that \ac{SEG} methods attain a $\bigoh(1/\sqrt{\nRunsNew})$ ergodic convergence rate.
By strengthening \cref{asm:mono} to \cref{asm:mono-strong}, we show that this result can be augmented in two synergistic ways:
under \cref{asm:solution,asm:Lipschitz,asm:mono-strong}, both the last iterate and the ergodic average of \ac{SEG} achieve a $\bigoh(1/\nRunsNew)$ convergence rate.
%Formally, we have:

\medskip
\begin{theorem}
\label{thm:stoch-global}
Suppose that $\vecfield$ satisfies \cref{asm:solution,asm:Lipschitz,asm:mono-strong},
and
assume that \eqref{eq:PEG} is run with
stochastic oracle feedback of the form \eqref{eq:oracle}
and
a step-size of the form $\current[\step] = \step / (\run + \strongstepm)$ for some $\step > 1/\strong$ and $\strongstepm \geq 4\lips\step$.
Then, the generated sequence of the algorithm's base states satisfies
\begin{equation}
\label{eq:rate-stoch-global}
\exof{\norm{\state_{\nRunsNew} - \sol}^{2}}
	\leq
	\frac{6\step^2\noisevar}{\strong\step-1}\frac{1}{\nRunsNew}
	+ \smalloh\left(\frac{1}{\nRunsNew}\right),
	%+ \smalloh\parens*{1/\nRunsNew}.
\end{equation}
while its ergodic average $\avg_{\nRunsNew} = \nRunsNew^{-1} \sum_{\runalt=\start}^\nRunsNew \current[\state][\runalt]$  enjoys the bound
%\YGH{The quick explanation is added in appendix and we do not need to explain here.}
\PM{Why $\bigoh(1/t)$? I think it would be better to write $\smalloh(\log t/t)$.
Also, I would suggest to write the $\bigoh(\dots)$ with $1/\cdots$, otherwise they stand out too much (because of the huge parentheses).}
\YGH{I prefer to stay with $\frac{1}{\run}$. Let's see what the others think.}
\begin{equation}
\label{eq:rate-stoch-global-erg}
    \exof{\norm{\avg_{\nRunsNew} - \sol}^{2}}
    \leq
	\frac{6\step^2\noisevar}{\strong\step-1}\frac{\log\nRunsNew}{\nRunsNew} + \smalloh\left(\frac{\log \nRunsNew}{\nRunsNew}\right). %+ \bigoh\left(\frac{1}{\nRunsNew}\right).
\end{equation}
\end{theorem}

Regarding our proof strategy for the last iterate of the process, we can no longer rely either on a contraction argument or the averaging mechanism that yields the $\bigoh(1/\sqrt{\nRunsNew})$ ergodic convergence rate.
Instead, we show in the appendix that $\current$ is (stochastically) quasi-Fejér in the sense of \cite{Com01,CP15};
then, leveraging the method's specific step-size, we employ successive numerical sequence estimates to control the summability error and obtain the $\bigoh(1/\nRunsNew)$ rate.

%----------------------------------------------------------------------
%%% LOCAL
%----------------------------------------------------------------------
\subsection{Local convergence}
\label{sec:stoch-local}

We proceed to examine the convergence of the method in the stochastic, \emph{non-monotone} case.
Our main result in this regard is the following.

\input{Figures/Illustration}

\begin{theorem}
\label{thm:stoch-local}
Let $\sol$ be a regular solution of \eqref{eq:SVI} and fix a tolerance level $\smallproba > 0$.
Suppose further that \eqref{eq:PEG} is run with
stochastic oracle feedback of the form \eqref{eq:oracle}
and
a variable step-size of the form $\current[\step] = \step / (\run + \strongstepm)$ for some $\step>1/\strong$ and large enough $\strongstepm$.
Then:%
\PM{Multiplied $\sadobj$ in the figure by $4$ to get rid of the fractions and save white space.}
\begin{enumerate}[label=\upshape(\itshape\alph*\upshape),leftmargin=2em]
%[label=\upshape(\itshape\alph*\upshape), leftmargin=2em]
\item
There are neighborhoods $\nhd$ and $\nhdalt$ of $\sol$ in $\points$ such that, if $\state_{\paststart}\in\nhd, \state_{\start}\in \nhdalt$, the event
\begin{equation}
\label{eq:event-horizon}
\event_{\infty}
	= \{\inter\in\nhd\,\text{for all $\run=\running$}\}
\end{equation}
occurs with probability at least $1-\smallproba$.
% \begin{equation}
% \label{eq:sol-Lyap}
% \probof{\event_{\infty} \given \state_{\paststart}\in\nhd, \state_{\start}\in\nhdalt}
% 	\geq 1 - \smallproba.
% \end{equation}
\item
%Provided that the event $\event_{\infty}$ occurs, we have
Conditioning on the above, we have:
\begin{equation}
\label{eq:rate-stoch-local}
\exof{\norm{\state_{\nRunsNew} - \sol}^{2} \given \event_{\infty}}
%\exof{\norm{\state_\nRunsNew-\sol}^2\given{\forall\run\in\intinterval{\laststart}{\nRunsNew-1}, \inter\in\nhd}}
	\leq \frac{4 \step^{2} (\vbound^{2} + \noisevar)}{(\strong\step - 1)(1 - \smallproba)} \frac{1}{\nRunsNew}
	+ \smalloh\left(\frac{1}{\nRunsNew}\right),
	%+ \smalloh\parens*{1/\nRunsNew}
\end{equation}
where
$\vbound = \sup_{\point\in\nhd} \norm{\vecfield(\point)} < \infty$
and
$\strong = \inf_{\point\in\nhd} \braket{\vecfield(\point)}{\point - \sol} / \norm{\point - \sol}^{2} > 0$.
\end{enumerate}
\end{theorem}

The finiteness of $\vbound$ and the positivity of $\strong$ are both consequences of the regularity of $\sol$ and their values only depend on the size of the neighborhood $\nhd$.
Taking a larger $\nhd$ would increase the algorithm's certified initialization basin but it would also negatively impact its convergence rate (since $\vbound$ would increase while $\strong$ would decrease).
Likewise, the neighborhood $\nhdalt$ only depends on the size of $\nhd$ and, as we explain in the appendix, it suffices to take $\nhdalt$ to be ``one fourth'' of $\nhd$.
%Finally, regarding the algorithm's step-size sequence, it suffices to have $\strong\step > 1$, while the choice of $\strongstepm$ depends on the tolerance level $\smallproba$.
%To guarantee a higher probability of convergence, one would want to take $\smallproba$ as small as possible;
%this would have a positive impact on the leading term of \eqref{eq:rate-stoch-local} \FI{I don't understand why}
%\YGH{I do not really understand this paragrahe. What do leading and subleading terms refer to?}
%but, at the same time, it would require a larger value for $\strongstepm$ \textendash\ which in turn would negatively affect the subleading term of \eqref{eq:rate-stoch-local}.
%We detail these issues in \cref{app:stoch-proofs}.

From the above, it becomes clear that the situation is significantly more involved than the corresponding deterministic analysis.
This is also reflected in the proof of \cref{thm:stoch-local} which requires completely new techniques, well beyond the straightforward localization scheme underlying \cref{thm:det-local}.
More precisely, a key step in the proof (which we detail in the appendix) is to show that
the iterates of the method remain close to $\sol$ for all $\run$ with arbitrarily high probability.
In turn, this requires showing that the probability of getting a string of ``bad'' noise realizations of arbitrary length is controllably small.
Even then however, the global analysis \emph{still} cannot be localized because conditioning changes the probability law under which the oracle noise is unbiased.
Accounting for this conditional bias requires a surprisingly delicate probabilistic argument which we also detail in the supplement.
%still cann is not applicable since the noise becomes biased conditioning on $\event_\infty$. 
%However, in the proof we demonstrate this question can be remedied by modifying the arguments accordingly to take into account the conditioning.

%% file: Figures/Illustration.tex
%----------------------------------------------------------------------
%%% NUMERICS
%----------------------------------------------------------------------
% !TEX root = ../Main.tex

\begin{figure*}[t]
\captionsetup[subfigure]{singlelinecheck=off}
\centering
\footnotesize
%% Subfigure 1
%----------------------------------------------------------------------
\begin{subfigure}[b]{.325\textwidth}
\includegraphics[width=\linewidth]{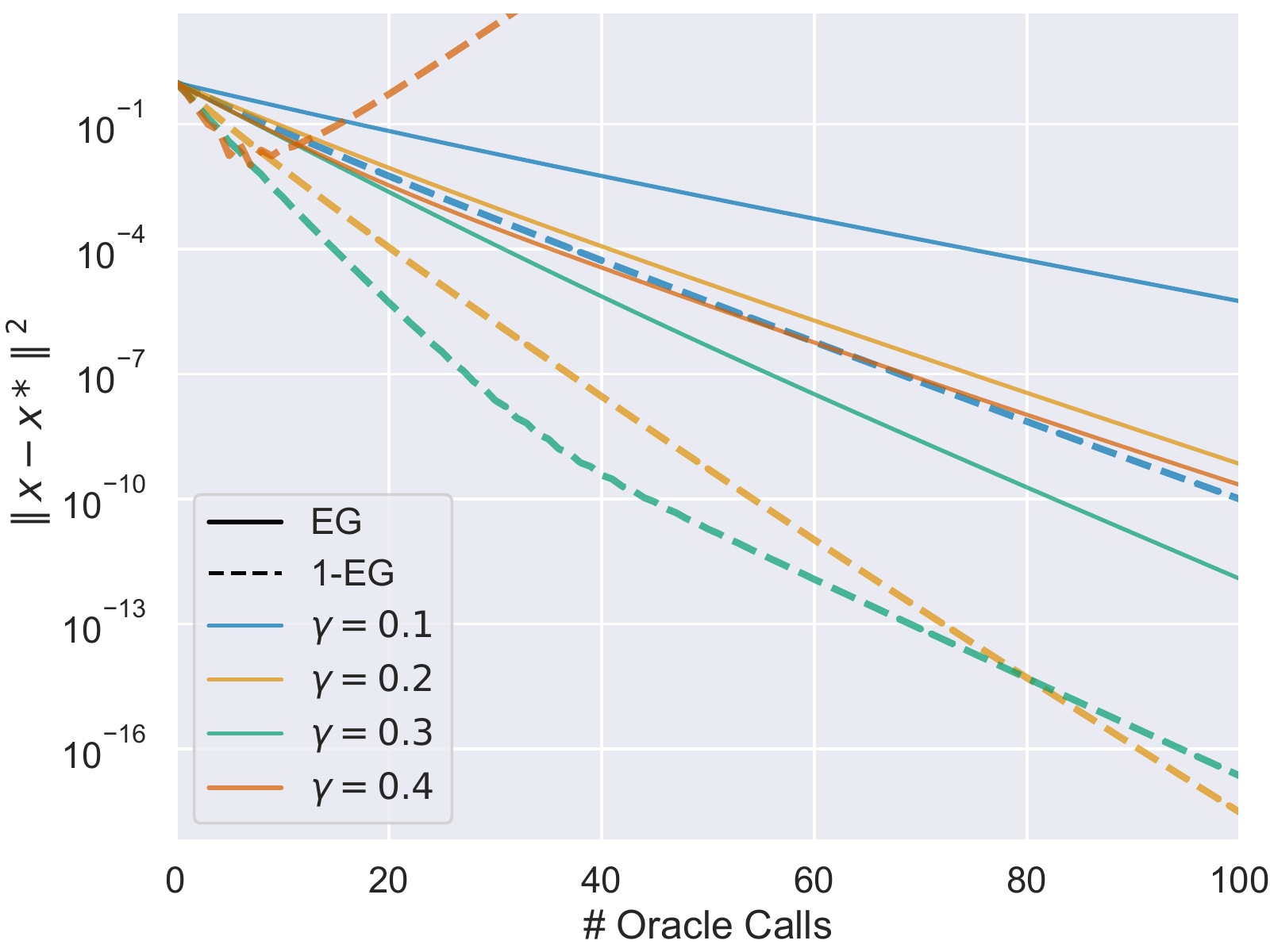}
\caption{%
\scriptsize
Strongly monotone {[$\indic_1\!=1$,  $\indic_2\!=0$]}, deterministic,
last iterate
%\begin{tabular}{l}
%Str. mon.
%	{[$\indic_1\!=1$,  $\indic_2\!=0$]}\\
%Deterministic\\
%Last iterate
%\end{tabular}
\label{fig:a}}
\end{subfigure}
%% Subfigure 2
%----------------------------------------------------------------------
\begin{subfigure}[b]{.325\textwidth}
\includegraphics[width=\linewidth]{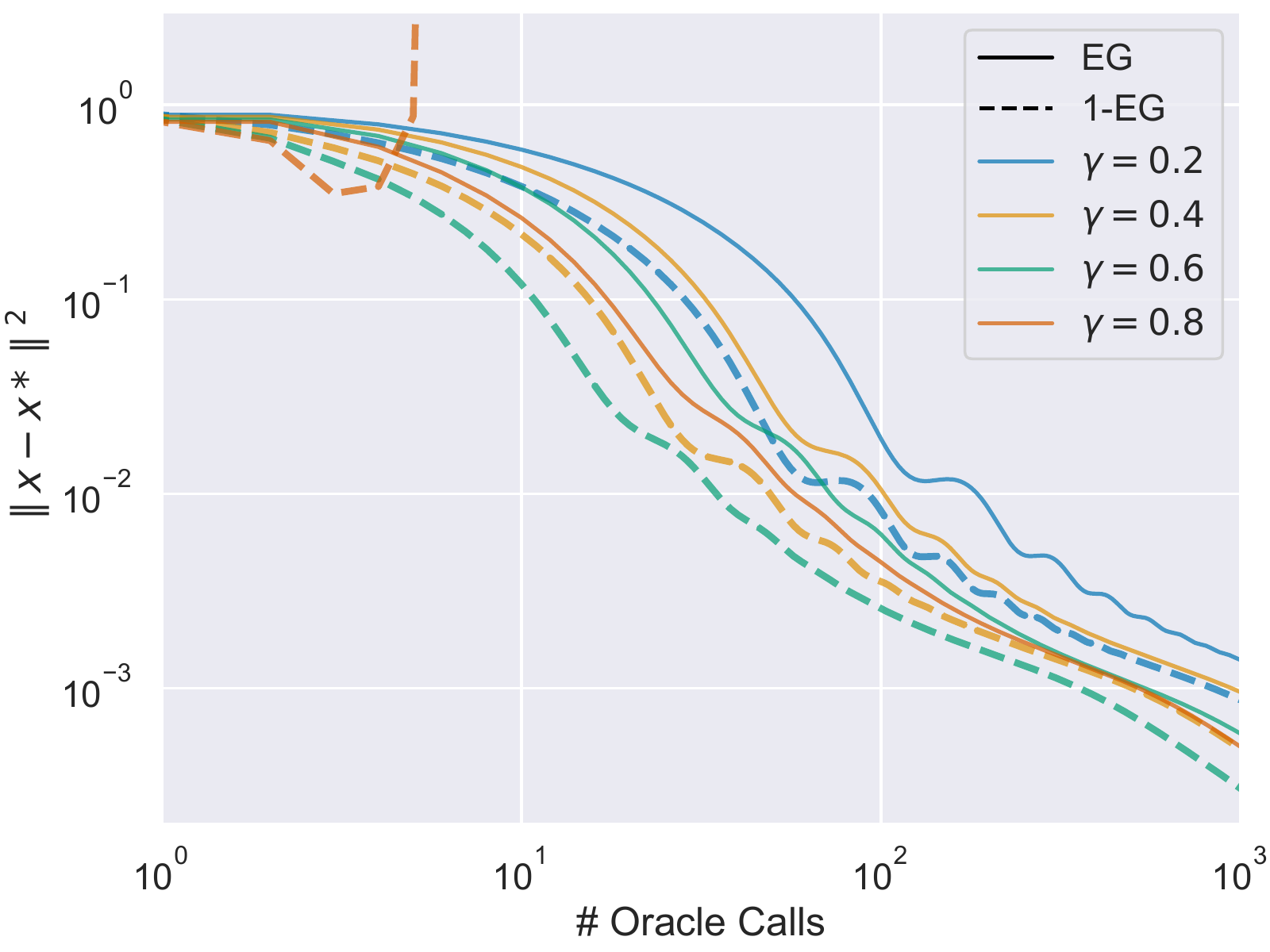}
\caption{%
\scriptsize
Monotone {[$\indic_1=0$,  $\indic_2=1$]}, deterministic,
ergodic averaging
%\begin{tabular}{l}
%Mon.
%	{[$\indic_1=0$,  $\indic_2=1$]}\\
%Deterministic\\
%Ergodic
%\end{tabular}
\label{fig:b}}
\end{subfigure}
%% Subfigure 3
%----------------------------------------------------------------------
\begin{subfigure}[b]{.325\textwidth}
\includegraphics[width=\linewidth]{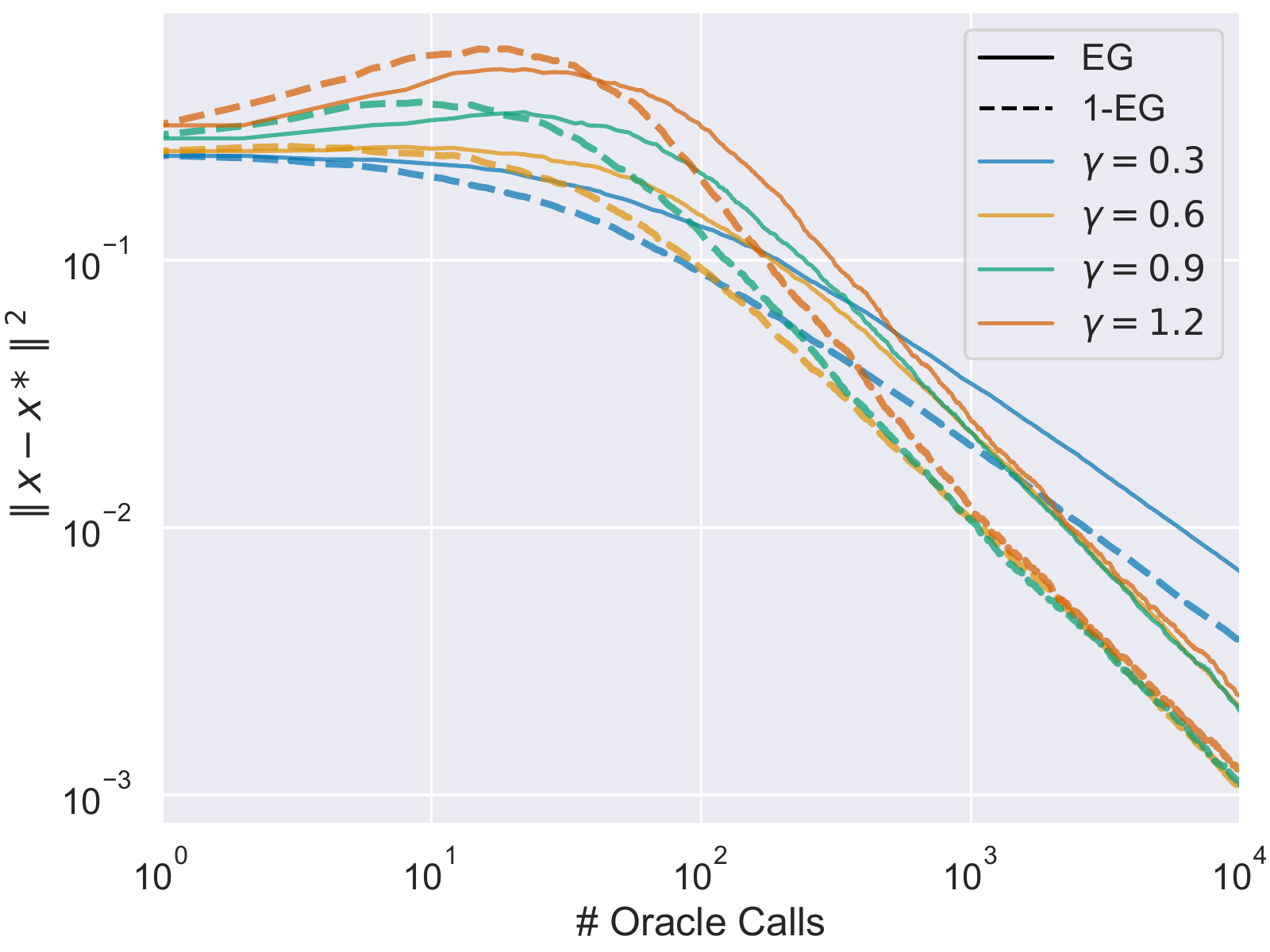}
\caption{%
\scriptsize
Non-monotone {[$\indic_1=1$, $\indic_2=-1$]}, iid {$\noise \sim \mathcal{N}(0,.01)$}, last iterate {($\strongstepm=15$)}
%\begin{tabular}{l}
%Non-monotone
%	{[$\indic_1=1$, $\indic_2=-1$]}\\
%Stochastic \acs{iid} {$\current[\noise] \sim \mathcal{N}(0,.01)$}\\
%Last iterate {($\strongstepm=15$)}
%\end{tabular}
\label{fig:c}}
\end{subfigure}%
%% Figure caption
%----------------------------------------------------------------------
\caption[]{Illustration of the performance of \ac{EG} and \ac{SEG} in the (a priori non-monotone) \acl{SP} problem
$$
\hspace*{-2.5cm}
\sadobj(\minvar,\maxvar)
	= 2\indic_1\minvar^\top\!\matmin_1\minvar
	+ \indic_2\big(\minvar^\top\!\matmin_2\minvar\big)^2
	- 2\indic_1\maxvar^\top\!\matmax_1\maxvar
	- \indic_2 \big(\maxvar^\top\!\matmax_2\maxvar\big)^2
	+ 4 \minvar^\top\!\matlin\maxvar
$$
on the full unconstrained space $\points=\vecspace=\R^{\vdim_1\times\vdim_2}$ with $\vdim_1=\vdim_2=1000$ and $ \matmin_1,\matmax_1, \matmin_2,\matmax_2 \succ 0$.
We choose three situations representative of the settings considered in the paper:
\begin{inparaenum}
[(\itshape a\upshape)]
\item
linear convergence of the last iterate of  the deterministic methods in strongly monotone problems;
\item
the $\bigoh(1/\nRunsNew)$ convergence of the ergodic average in monotone, deterministic problems;
and
\item
the $\bigoh(1/\nRunsNew)$ local convergence rate of the method's last iterate in stochastic, \emph{non-monotone} problems.
\end{inparaenum}
For (\emph{a}) and (\emph{b}), the origin is the unique solution of \eqref{eq:SVI}, and for (\emph{c}) it is a regular solution thereof.
We observe that \ac{SEG} consistently outperforms \ac{EG} in terms of oracle calls for a fixed step-size, and the observed rates are consistent with the rates reported in \cref{tab:overview}.}
\label{fig:ills}
\end{figure*}

%% file: Conclusions.tex
%----------------------------------------------------------------------
%%% CONCLUSIONS
%----------------------------------------------------------------------
% !TEX root = ./Main.tex

Our aim in this paper was to provide a synthetic view of single-call surrogates to the \acl{EG} algorithm, and to establish optimal convergence rates in a range of different settings \textendash\ deterministic, stochastic, and/or non-monotone.
Several interesting avenues open up as a result, from extending the theory to more general Bregman proximal settings, to developing an adaptive version as in the recent work \cite{BL19} for two-call methods.
We defer these research directions to future work.

%% file: Thanks.tex
%----------------------------------------------------------------------
%%% THANKS
%----------------------------------------------------------------------
% !TEX root = ./Main.tex
%
%
This work benefited from financial support by MIAI Grenoble Alpes (Multidisciplinary Institute in Artificial Intelligence).
P.~Mertikopoulos was partially supported by
the French National Research Agency (ANR) grant ORACLESS (ANR\textendash 16\textendash CE33\textendash 0004\textendash 01)
and
the EU COST Action CA16228 ``European Network for Game Theory'' (GAMENET)\afterhead

%% file: App-TechLemmas.tex
%----------------------------------------------------------------------
%%% APP: Technical Lemmas
%----------------------------------------------------------------------
% !TEX root = ./Main.tex

\begin{lemma}
\label{lem:3points}
Let %$\vdim \ge 1$,
%\FI{I think $\vdim \ge 1$ can be safely dropped (so I dropped it).}
$\point, \dpoint \in \vecspace$ and $\cvx\subseteq\vecspace$ be a closed convex set.
We set $\pointnew \defeq \proj_{\cvx}(\point-\dpoint)$.
For all $\arpoint\in\cvx$, we have
\begin{equation}
  \norm{\pointnew - \arpoint}^2
  \le
  \norm{\point - \arpoint}^2
  - 2 \product{\dpoint}{\pointnew-\arpoint}
  - \norm{\pointnew - \point}^2.
\end{equation}
\end{lemma}

\begin{proof}
Since $\arpoint\in\cvx$, we have the following property $\product{\pointnew-(\point-y)}{\pointnew-\arpoint} \le 0$, leading to
\begin{align}
  \norm{\pointnew - \arpoint}^2
  & = \norm{\pointnew - \point + \point - \arpoint}^2
  \notag\\
  & = \norm{\point - \arpoint}^2
  + 2 \product{\pointnew-\point}{\point-\arpoint}
  + \norm{\pointnew - \point}^2
  \notag\\
  & = \norm{\point - \arpoint}^2
  + 2 \product{\pointnew-\point}{\pointnew-\arpoint}
  - \norm{\pointnew - \point}^2
  \notag\\
  & \le
  \norm{\point - \arpoint}^2
  - 2 \product{\dpoint}{\pointnew-\arpoint}
  - \norm{\pointnew - \point}^2.
  \qedhere
%  \\[-0.5em]
%  \tag*{\qedhere}
\end{align}
\end{proof}

\begin{lemma}
\label{lem:4points}
%\YGH{Damn it. The proof was wrong. The new statement is not so natural but like this I only need to change a minimum of things.}
Let %$\vdim \ge 1$, \JM{same remark here.} 
$\point, \dpoint_1, \dpoint_2\in\vecspace$
and $\cvx_1, \cvx_2\subseteq\vecspace$ be two closed convex sets.
We set $\pointnew_1 \defeq \proj_{\cvx_1}(\point-\dpoint_1)$ and
$\pointnew_2 \defeq \proj_{\cvx_2}(\point-\dpoint_2)$.
\begin{enumerate}[wide, label=\upshape (\alph*)] % leftmargin=*
    \item \label{lem:4points-a}
    If $\cvx_2=\vecspace$, for all $\arpoint\in\vecspace$, it holds
    \begin{equation}\label{eq:4points-a}
        \norm{\pointnew_2 - \arpoint}^2
        = \norm{\point-\arpoint}^2 - 2\product{\dpoint_2}{\pointnew_1-\arpoint}
        + \norm{\pointnew_2 - \pointnew_1}^2 - \norm{\pointnew_1 - \point}^2.
    \end{equation}
    \item \label{lem:4points-b}
    If $\cvx_2\subseteq\cvx_1$, for all $\arpoint\in\cvx_2$, it holds
    \begin{align}
        \norm{\pointnew_2 - \arpoint}^2
        &\le
        \norm{\point-\arpoint}^2 - 2\product{\dpoint_2}{\pointnew_1-\arpoint}
        + 2\product{\dpoint_2-\dpoint_1}{\pointnew_1-\pointnew_2}
        \notag\\
        &\enspace - \norm{\pointnew_2 - \pointnew_1}^2 - \norm{\pointnew_1 - \point}^2
        \notag\\
        &\le
        \norm{\point-\arpoint}^2 - 2\product{\dpoint_2}{\pointnew_1-\arpoint}
        + \norm{\dpoint_2 - \dpoint_1}^2 - \norm{\pointnew_1 - \point}^2.
    \label{eq:4points-b}
    \end{align}
\end{enumerate}

\end{lemma}

\begin{proof}
\begin{enumerate}[wide, label=(\alph*)]
\item
% Since $\cvx_1\subseteq\cvx_2$ and $\pointnew_1\in\cvx_1$, we have also $\pointnew_1\in\cvx_2$.
% \autoref{lem:3points} applied to
% $(\point, \dpoint, \pointnew, \arpoint, \cvx)
% \subs (\point, \dpoint_2, \pointnew_2, \arpoint, \cvx_2)$
% and
% $(\point, \dpoint, \pointnew, \arpoint, \cvx)
% \subs (\point, \dpoint_2, \pointnew_2, \pointnew_1, \cvx_2)$
% gives
% \begin{gather}
%   \label{eq:4points-a1}
%   \norm{\pointnew_2 - \arpoint}^2
%   \le
%   \norm{\point - \arpoint}^2
%   - 2 \product{\dpoint_2}{\pointnew_2-\arpoint}
%   - \norm{\pointnew_2 - \point}^2, \\
%   \label{eq:4points-a2}
%   \norm{\pointnew_2 - \pointnew_1}^2
%   \le
%   \norm{\point - \pointnew_1}^2
%   - 2 \product{\dpoint_2}{\pointnew_2- \pointnew_1}
%   - \norm{\pointnew_2 - \point}^2.
% \end{gather}
% Subtracting \eqref{eq:4points-a2} from \eqref{eq:4points-a1},
% we obtain \eqref{eq:4points-a}.

We develop
\begin{align}
    \norm{\pointnew_2-\arpoint}^2
    &=\norm{\pointnew_2-\pointnew_1+\pointnew_1-\point+\point-\arpoint}^2
    \notag\\
    &=\norm{\pointnew_2-\pointnew_1}^2+\norm{\pointnew_1-\point}^2+\norm{\point-\arpoint}^2
    \notag\\
    &\enspace
    + 2\product{\pointnew_2-\pointnew_1}{\pointnew_1-\arpoint}
    + 2\product{\pointnew-\point}{\point-\arpoint}
    \notag\\
    &=\norm{\pointnew_2-\pointnew_1}^2-\norm{\pointnew_1-\point}^2+\norm{\point-\arpoint}^2
    \notag\\
    &\enspace
    + 2\product{\pointnew_2-\pointnew_1}{\pointnew_1-\arpoint}
    + 2\product{\pointnew_1-\point}{\pointnew_1-\arpoint}
    \notag\\
    &=\norm{\point-\arpoint}^2 - 2\product{\dpoint_2}{\pointnew_1-\arpoint}
    + \norm{\pointnew_2 - \pointnew_1}^2 - \norm{\pointnew_1 - \point}^2,
\end{align}
where in the last line we use $\pointnew_2-\point=-\dpoint_2$ since $\cvx_2=\vecspace$.

\item
With $\pointnew_2\in\cvx_2\subseteq\cvx_1$, we can apply \autoref{lem:3points} to
$(\point, \dpoint, \pointnew, \arpoint, \cvx)
\subs (\point, \dpoint_2, \pointnew_2, \arpoint, \cvx_2)$
and
$(\point, \dpoint, \pointnew, \arpoint, \cvx)
\subs (\point, \dpoint_1, \pointnew_1, \pointnew_2, \cvx_1)$, which yields
\begin{gather}
  \label{eq:4points-b1}
  \norm{\pointnew_2 - \arpoint}^2
  \le
  \norm{\point - \arpoint}^2
  - 2 \product{\dpoint_2}{\pointnew_2-\arpoint}
  - \norm{\pointnew_2 - \point}^2, \\
  \label{eq:4points-b2}
  \norm{\pointnew_1 - \pointnew_2}^2
  \le
  \norm{\point -  \pointnew_2}^2
  - 2 \product{\dpoint_1}{\pointnew_1- \pointnew_2}
  - \norm{\pointnew_1 - \point}^2.
\end{gather}
By summing \eqref{eq:4points-b1} and \eqref{eq:4points-b2}, we readily get the
first inequality of \eqref{eq:4points-b}.
We conclude with help of Young's inequality
$2\product{\dpoint_2-\dpoint_1}{\pointnew_1-\pointnew_2} \le
\norm{\dpoint_2-\dpoint_1}^2 + \norm{\pointnew_1-\pointnew_2}^2$.
\qedhere
\end{enumerate}
\end{proof}

\begin{lemma}[{{\citet[Lemma~1]{Chu54}}}]
\label{lem:chung1954}
Let $\seqinf{\seqitem}{\run}$ be a sequence of real numbers
and $\strongstepm, \run_0\in\N$
such that for all $\run\ge\run_0$,
\begin{equation}
    \label{eq:chung_rec}
    \update[\seqitem]
    \le
    \left(1-\frac{\consc}{\run+\strongstepm}\right)\current[\seqitem]
    + \frac{\conscalt}{(\run+\strongstepm)^2},
\end{equation}
where $\consc > 1$ and $\conscalt > 0$.
Then,
%\YGH{I feel it is better to put $\run$ here though in the stochastic proofs I note $\nRuns$ for the convergence result.}
\begin{equation}
    \label{eq:chung_bound}
    \seqitem_\run
    \le
    \frac{\conscalt}{\consc-1}\frac{1}{\run} + \smalloh\left(\frac{1}{\run}\right).
\end{equation}
% T version
% \begin{equation}
%     \label{eq:chung_bound}
%     \seqitem_\nRuns
%     \le
%     \frac{\conscalt}{\consc-1}\frac{1}{\nRuns} + \smalloh\left(\frac{1}{\nRuns}\right).
% \end{equation}
%\YGH{Are you ok with the $\seqitem_{\run+\strongstepm}\subs\ldots$ that I put in the proof later? Otherwise I need to formulate differently.}
\end{lemma}
\begin{proof}
For the sake of completeness, we provide a basic proof for the above lemma (which is a direct corollary of \citet[Lemma~1]{Chu54}).
Let $\consc > 1$ and $\intg\in\N$, we have
\begin{equation}
    \frac{1}{\intg+1}
    - \left(1-\frac{\consc}{\intg}\right)\frac{1}{\intg}
    =
    \frac{\consc}{\intg^2}
    - \left(\frac{1}{\intg} - \frac{1}{\intg+1}\right)
    = \frac{\consc-1}{\intg^2} + \frac{1}{\intg^2(\intg+1)}.
\end{equation}
This shows that for any $\conscalt>0$
\begin{equation}
    \label{eq:chung_basic_ineq}
    \frac{\conscalt}{\consc-1}
    \left(
        \frac{1}{\intg+1}
        - \left(1-\frac{\consc}{\intg}\right)\frac{1}{\intg}
    \right) = \frac{\conscalt}{\intg^2} + \frac{\conscalt}{\intg^2(\intg+1)(\consc-1)} \ge \frac{\conscalt}{\intg^2}.
\end{equation}
By substituting $\intg\subs\run+\strongstepm$,
\eqref{eq:chung_rec} combined with \eqref{eq:chung_basic_ineq} yields
\begin{equation}
    \label{eq:chung_rec_long}
    \update[\seqitem] - \frac{\conscalt}{\consc-1} \frac{1}{\run+\strongstepm+1}
    \le
    \left(1-\frac{\consc}{\run+\strongstepm}\right)
    \left(\current[\seqitem]
    - \frac{\conscalt}{\consc-1} \frac{1}{\run+\strongstepm}\right).
\end{equation}
Let us define
$\current[\seqitemalt] \defeq \current[\seqitem] - \conscalt/((\consc-1)(\run+\strongstepm))$.
\eqref{eq:chung_rec_long} becomes
\begin{equation}
    \label{eq:chung_rec_short}
    \update[\seqitemalt]
    \le
    \left(1-\frac{\consc}{\run+\strongstepm}\right)
    \current[\seqitemalt].
\end{equation}
This inequality holds for all $\run\ge\run_0$. Then, either: \\
\textbullet~  $\current[\seqitemalt]$ becomes non-positive for some
$\run > \run_1 = \max(\run_0, \floor{\consc}-\strongstepm)$, and \eqref{eq:chung_rec_short} implies that this is also the case for all subsequent $\run$, which leads to 
%subsequent $\run$, namely
%
\begin{equation}
    \current[\seqitem][\run] \le \frac{\conscalt}{\consc-1} \frac{1}{\run+\strongstepm}.
\end{equation}

\noindent \textbullet~  or $\current[\seqitemalt]$ is positive for all $\run > \run_1$ and we get 
\begin{equation}
    0 < \seqitemalt_\run
    \le
    \seqitemalt_{\run_1}\prod_{\runalt=\run_1}^{\run-1}
    \left(1-\frac{\consc}{\runalt+\strongstepm}\right)
    = \bigoh\left(\frac{1}{\run^\consc}\right)
    = \smalloh\left(\frac{1}{\run}\right).
\end{equation}
In both cases, \eqref{eq:chung_bound} is verified.
\end{proof}

% T version
% \textbullet~  $\current[\seqitemalt]$ becomes non-positive for some
% $\run > \run_1 = \max(\run_0, \floor{\consc}-\strongstepm)$, and \eqref{eq:chung_rec_short} implies that this is also the case for all $\nRuns \ge \run$, which leads to 
% %subsequent $\run$, namely
% %
% \begin{equation}
%     \current[\seqitem][\nRuns] \le \frac{\conscalt}{\consc-1} \frac{1}{\nRuns+\strongstepm}.
% \end{equation}
% %

% \noindent \textbullet~  or $\current[\seqitemalt]$ is positive for all $\nRuns > \run_1$ and we get 
% %
% \begin{equation}
%     0 < \seqitemalt_\nRuns
%     \le
%     \seqitemalt_{\run_1}\prod_{\run=\run_1}^{\nRuns-1}
%     \left(1-\frac{\consc}{\run+\strongstepm}\right)
%     = \bigoh\left(\frac{1}{\nRuns^\consc}\right)
%     = \smalloh\left(\frac{1}{\nRuns}\right).
% \end{equation}

\begin{lemma}
\label{lem:regualr}
Let $\sol$ be a regular solution of \eqref{eq:SVI}. Then, there exists constants $\nhdradius, \strong, \lips > 0$
such that $\vecfield$ is $\lips$-Lipschitz continuous on $\cpt\defeq\ball_{\nhdradius}(\sol)$
and $\product{\vecfield(\point)}{\point-\sol}\ge\strong\norm{\point-\sol}^2$ for all $\point\in\nhd\defeq\points\intersect\cpt$.
\end{lemma}
\begin{proof}
The Lipschitz continuity is straightforward:
a $C^{1}$-smooth operator is necessarily locally Lipschitz and thus Lipshitz on every compact. The proof consists in establishing the existence of $\strong$. To this end, we consider the following function:
\begin{equation}
\begin{array}{crcl}%*3{>{\displaystyle}c}}
\matfunc \colon & \R^{\vdim\times\vdim} & \longrightarrow & \R \\
                & \mat                  & \longmapsto     & \min_{\tvec\in\tcone_{\!\points}(\sol),\norm{\tvec}=1}\tvec^\top\mat\tvec
\end{array}
\end{equation}
where $\tcone_\points(\sol)$ denotes the tangent cone to $\points$ at $\sol$.
The function $\matfunc$ is concave as it is defined as a pointwise minimum over a set of linear functions.
This in turn implies the continuity $\matfunc$ because every concave function is continous on the interior of its effective domain.
%\matfunc$ is defined on the whole $\R^{\vdim\times\vdim}$, it follows that $\matfunc$ is continuous.
The solution $\sol$ being regular, we have $\matfunc(\Jacf{\vecfield}{\sol})>0$.
Combined with the continuity of $\Jac_\vecfield$ in a neighborhood of $\sol$, we deduce the existence of $\nhdradius, \strong>0$
such that $\matfunc(\Jacf{\vecfield}{\point})\ge\strong$ for all $\point\in\cpt=\ball_{\nhdradius}(\sol)$.
Now let $\point\in\nhd=\points\intersect\cpt$. %\ballr{\nhdradius}{\sol}$.
It holds:
\begin{equation}
    \vecfield(\point) - \vecfield(\sol) =
    \left(\int_0^1 \Jacf{\vecfield}{\sol+\scalar(\point-\sol)}\dd\scalar\right)(\point-\sol).
\end{equation}
% By mean value theorem, there exist $\scalar\in(0,1)$ such that
% \begin{equation}
%     \vecfield(\point) - \vecfield(\sol) = \Jacf{\vecfield}{\sol+\scalar(\point-\sol)}(\point-\sol).
% \end{equation}
Consequently, writing $\tvec=\point-\sol \in \tcone_\points(\sol)$, $\pointalt_\scalar=\sol+\scalar(\point-\sol)\in\cpt$, we have
\begin{align}
    &\product{\vecfield(\point)-\vecfield(\sol)}{\point-\sol}
    =
    \tvec^\top\left(\int_0^1 \Jacf{\vecfield}{\pointalt_\scalar}\dd\scalar\right)\tvec\\
    &~~~~~~~~~~~~~~~~~~~~~~~~~~~~~~~~~~~\ge
    \left(\int\matfunc(\Jacf{\vecfield}{\pointalt_\scalar})\dd\scalar\right) \norm{\tvec}^2 \ge \alpha\norm{\tvec}^2 = \alpha\norm{\point-\sol}^2.
\end{align}
% \begin{align}
%     \product{\vecfield(\point)-\vecfield(\sol)}{\point-\sol} = \tvec^\top\Jacf{\vecfield}{\pointalt}\tvec
%     \ge\matfunc(\Jacf{\vecfield}{\pointalt}) \norm{\tvec}^2 \ge \alpha\norm{\tvec}^2 = \alpha\norm{\point-\sol}^2.
% \end{align}
Finally, since $\sol$ is a solution of \eqref{eq:SVI}, we have $\product{\vecfield(\sol)}{\point-\sol}\ge0$
and 
\begin{align}
\product{\vecfield(\point)}{\point-\sol}\ge\product{\vecfield(\point)-\vecfield(\sol)}{\point-\sol}\ge\alpha\norm{\point-\sol}^2.%\\[-0.5em]
  %\tag*{\qedhere}
\end{align}
This ends the proof.
\end{proof}

%% file: App-DetProofs.tex
%----------------------------------------------------------------------
%%% APP: DETERMINISTIC PROOFS
%----------------------------------------------------------------------
% !TEX root = ./Main.tex

\subsection{Proof of \autoref{lem:descent}}

%\LemDescent*

%\begin{proof}[Proof of \autoref{lem:descent}]

In the definition of $\reserr$, instead of taking $\points_\radius = \points\intersect\ballr{0}{\radius}$ we consider $\points_\radius = \points\intersect\ballr{\state_\start}{\radius}$.
%and we choose $\radius$ large enough so that $\points_\radius$ is not empty.
Summing \eqref{eq:descent} over $\runalt$ and rearranging the term leads to
\begin{equation}
\label{eq:descent_tele}
    \sum_{\runalt=\start}^{\run}
    2\scalar_\runalt\product{\vecfield(\inter[\state][\runalt])}{\inter[\state][\runalt]-\arpoint}
    \le \norm{\state_{\start}-\arpoint}^2 - \norm{\update -\arpoint}^2 + \tele_\start - \update[\tele]
    \le \norm{\state_{\start}-\arpoint}^2 + \tele_\start.
\end{equation}
For any $\arpoint\in\points_\radius$, we have
$\norm{\state_\start-\arpoint}^2\le\radius^2$, and by monoticity of $\vecfield$,
\begin{equation}
    \product{\vecfield(\arpoint)}{\inter[\state][\runalt]-\arpoint} \le \product{\vecfield(\inter[\state][\runalt])}{\inter[\state][\runalt]-\arpoint}.   
\end{equation}
In other words, for all $\arpoint\in\points_\radius$,
\begin{equation}
    2\sum_{\runalt=\start}^{\nRunsNew}\current[\scalar][\runalt]\product{\vecfield(\arpoint)}{\inter[\state][\runalt]-\arpoint}
    \le \radius^2 + \tele_\start.
    \label{eq:ergodic_smaller}
\end{equation}
Dividing the two sides of \eqref{eq:ergodic_smaller} by $2\sum_{\runalt=\start}^{\run}\current[\scalar][\runalt]$ and
maximizing over $\arpoint\in\points_\radius$ leads to the desired result.
%We conclude the proof by dividing \eqref{eq:descent_tele} by $2\sum_{\runalt=\start}^{\run}\current[\scalar]$,
%maximizing over $\arpoint\in\points\intersect\ballr{\state_\start}{\radius}$
%and applying \autoref{lem:err_bound}.
%\end{proof}

% version T
% \begin{equation}
% \label{eq:descent_tele}
%     \sum_{\run=\start}^{\nRuns}
%     2\scalar_\run\product{\vecfield(\inter)}{\inter-\arpoint}
%     \le \norm{\state_{\start}-\arpoint}^2 - \norm{\update -\arpoint}^2 + \tele_\start - \update[\tele]
%     \le \norm{\state_{\start}-\arpoint}^2 + \tele_\start.
% \end{equation}

% =================================================
\subsection{Proof of \cref{thm:det-global-erg}}
\label{app:det-ergodic-proof}

%\ThDetErgodic*

%\begin{proof}
%Our proof is mainly based on \autoref{lem:4points}, whereas the calculation details differ from one algorithm to another.
%We therefore present the three proofs separately.
%\YGH{Need to be revised to match the new statement.}
%Below we %present the proofs
%prove the theorem for the three algorithms separately.
%For sake of simplicity, we use a different initialization from the one indicated in \cref{alg:MG}.
%We suppose that for \eqref{eq:PEG} and \eqref{eq:OG} $\state_\paststart$ and $\state_\start$
%are initialized randomly
%while for \eqref{eq:RG} we initialize with random $\state_\laststart$ and $\state_\paststart$ and set
%$\state_\start=\proj_\points(\state_\laststart-\step\vecfield(\state_\paststart))$.
%\YGH{I indicate here that we use different initialization rules from the one in main. Idem for stochastic part.}

To facilitate analysis and presentation of our results, 
%we adopt in what follows  a different initialization from the one indicated in \cref{alg:MG}.
\eqref{eq:PEG} and \eqref{eq:OG} are initialized with random $\state_\paststart$ and $\state_\start$ in $\points$
while for \eqref{eq:RG} we start with $\state_\laststart$ and $\state_\paststart$.
We are constrained to have different initial states in \eqref{eq:RG} due to its specific formulation.

The theorem is immediate from \cref{lem:descent} if we know that \eqref{eq:descent} is verified by the generated iterates for some $\seqinf{\scalar}{\run}, \seqinf{\tele}{\run}\in\R_+^\N$.
Below, we show it separately for \ac{PEG}, \ac{OG} and \ac{RG} under \cref{asm:Lipschitz} and with $\step$ selected as per the theorem statement.
Moreover, we have $\seqinf{\scalar}{\run} \equiv \step$ and $\tele_\start\le\norm{\state_\start-\state_\paststart}^2$ for all methods, hence the corresponding bound in our statement. The arguments used in the proof are inspired from \cite{Tse00,Mal15,GBVV+19} but we emphasize the relation between the analyses of these algorithms by putting forward the technical \autoref{lem:4points}.

%To leverage \cref{lem:descent}, we show that 
%the generated iterates of the algorithm satisfy \eqref{eq:descent} for $\seqinf{\scalar}{\run} \equiv \step$
%and some $\seqinf{\tele}{\run}\in\R_+^\N$ satisfying $\tele_\start\le\norm{\state_\start-\state_\paststart}^2$.

%The techniques that are used here are not totally new.
%Nonetheless, we unify the different proofs by putting forward \autoref{lem:4points}.
%different proofs transparent by putting forward \autoref{lem:4points}.
%present them in a way so that they can all be easily deduced from \autoref{lem:4points}.

%\textbf{\acf{PEG}.}\enspace
\Paragraph{\acf{PEG}\afterhead}
For $\run \ge \start$,
the second inequality of \autoref{lem:4points} \ref{lem:4points-b} applied to 
$(\point, \dpoint_1, \dpoint_2, \pointnew_1, \pointnew_2, \cvx_1, \cvx_2)
\subs(\current, \step\vecfield(\past), \step\vecfield(\inter), \inter, \update, \points, \points)$
results in
\begin{align}
  \norm{\update - \arpoint}^2
  & \le
  \norm{\current-\arpoint}^2
  - 2\step\product{\vecfield(\inter)}{\inter-\arpoint}
  \notag\\
  & \enspace
  + \step^2\norm{\vecfield(\inter) - \vecfield(\past)}^2
  - \norm{\inter - \current}^2
  \notag\\
  &  \le
  \norm{\current-\arpoint}^2
  - 2\step\product{\vecfield(\inter)}{\inter-\arpoint}
  \notag\\
  & \enspace
  + \step^2\lips^2\norm{\inter - \past}^2
  - \norm{\inter - \current}^2
  \label{eq:peg_det_4points}
\end{align}
where we used the fact that $\vecfield$ is $\lips$-Lipschitz continuous for the second inequality.
% We claim that for all $\run\ge\afterstart$,
% %
% \begin{equation}
% \label{eq:peg_det_trick}
%     \norm{\inter - \past}^2
%     \le 4\norm{\inter - \current}^2
%     + \norm{\past - \pastpast}^2 - \norm{\inter - \past}^2.
% \end{equation}
%

Now, let us use Young's inequality $\norm{a+b}^2 \le 2\norm{a}^2 + 2\norm{b}^2$ to get
\begin{equation}
\label{eq:peg_det_1}
  \norm{\inter - \past}^2
  \le 2\norm{\inter - \current}^2 + 2\norm{\current - \past}^2
\end{equation}
and the non-expansiveness of the projection to get for any $\run\ge\afterstart$,
\begin{equation}
\label{eq:peg_det_2}
  \norm{\current - \past}^2
  \le \norm{\last - \step\vecfield(\past)
        -\last + \step\vecfield(\pastpast)}^2
  \le \step^2\lips^2\norm{\past-\pastpast}^2.
\end{equation}
Combining \eqref{eq:peg_det_1} and \eqref{eq:peg_det_2}, we obtain
% %
% \begin{align}
%     \norm{\inter - \past}^2
%     & = 2\norm{\inter - \past}^2 - \norm{\inter - \past}^2
%  \notag\\
%     & \le 4\norm{\inter - \current}^2 + 4\norm{\current - \past}^2
%     - \norm{\inter - \past}^2
%  \notag\\
%     & \le 4\norm{\inter - \current}^2
%     + 4\step^2\lips^2\norm{\past - \pastpast}^2 - \norm{\inter - \past}^2
%  \notag\\
%     & \le 4\norm{\inter - \current}^2
%     + \norm{\past - \pastpast}^2 - \norm{\inter - \past}^2.
%     \label{eq:peg_det_trick}
% \end{align}
% %
%
\begin{align}
    \norm{\inter - \past}^2
    & \le 2\norm{\inter - \current}^2
    + 2\step^2\lips^2\norm{\past - \pastpast}^2 
  \\
    & \le 2\norm{\inter - \current}^2
    +  \frac{1}{2}\norm{\past - \pastpast}^2,
    \label{eq:peg_det_trick_prev}
\end{align}
where we used the fact that $\step\le1/(2\lips)$ in the last inequality; and in order to display a telescopic term, we reformulate \eqref{eq:peg_det_trick_prev} as 
\begin{align}
    \norm{\inter - \past}^2
    &= 2 \norm{\inter - \past}^2 - \norm{\inter - \past}^2
  \notag\\
    %&~~~~~~~~~~~~~~~
    &\le 4\norm{\inter - \current}^2
    + \norm{\past - \pastpast}^2 - \norm{\inter - \past}^2.
    \label{eq:peg_det_trick}
\end{align}

We now substitute \eqref{eq:peg_det_trick} in \eqref{eq:peg_det_4points} to get for all $\run\ge\afterstart$,
\begin{align}
  \norm{\update-\arpoint}^2
  %\\
  %&\enspace\le
  %\norm{\current-\arpoint}^2
  %-2\step\product{\vecfield(\inter)}{\inter-\arpoint}
  %+ \step^2\lips^2\norm{\inter - \past}^2
  %- \norm{\inter - \current}^2
  %\notag\\
  &\le
  \norm{\current-\arpoint}^2
  -2\step\product{\vecfield(\inter)}{\inter-\arpoint}
  + (4\step^2\lips^2-1)\norm{\inter - \current}^2
  \notag\\
  &\enspace\enspace
  + \step^2\lips^2(\norm{\past - \pastpast}^2 - \norm{\inter - \past}^2)
  \notag\\
 &\le
  \norm{\current-\arpoint}^2
  -2\step\product{\vecfield(\inter)}{\inter-\arpoint}
  \notag\\
  &\enspace\enspace
  + \step^2\lips^2(\norm{\past - \pastpast}^2 - \norm{\inter - \past}^2), %\\
  %\label{eq:peg_det_tele}
\end{align}
and thus \eqref{eq:descent} holds true for all $\run\ge\afterstart$ with $\scalar_\run=\step$ and $\tele_\run=\step^2\lips^2\norm{\past - \pastpast}^2$.

Finally, for $\run=\start$, we have
\begin{align}
    &\step^2\lips^2\norm{\state_\interstart - \state_\paststart}^2
    - \norm{\state_\interstart - \state_\start}^2
    \notag\\
    &\enspace\le
    4\step^2\lips^2\norm{\state_\interstart - \state_\start}^2
    + 4\step^2\lips^2\norm{\state_\start - \state_\paststart}^2
    - \step^2\lips^2\norm{\state_\interstart - \state_\paststart}^2
    - \norm{\state_\interstart - \state_\start}^2
    \notag\\
    &\enspace\le
    4\step^2\lips^2\norm{\state_\start - \state_\paststart}^2
    -\step^2\lips^2\norm{\state_\interstart - \state_\paststart}^2,
   % \label{eq:peg_det_init}
\end{align}
which, plugged into \eqref{eq:peg_det_4points} gives
\begin{align}
&\norm{\state_\afterstart - \arpoint}^2
    \le
  \norm{\state_\start-\arpoint}^2
  - 2\step\product{\vecfield(\state_\interstart)}{\state_\interstart-\arpoint} 
  + \step^2\lips^2\norm{\state_\interstart - \state_\paststart}^2
  - \norm{\state_\interstart - \state_\start}^2
  \notag\\
    &~~~  \le
  \norm{\state_\start-\arpoint}^2
  - 2\step\product{\vecfield(\state_\interstart)}{\state_\interstart-\arpoint} 
  +  4\step^2\lips^2\norm{\state_\start - \state_\paststart}^2
    -\step^2\lips^2\norm{\state_\interstart - \state_\paststart}^2 \label{eq:peg_det_init}
\end{align}
which also matches \eqref{eq:descent} for $\run=\start$ with $\scalar_\run=\step$, $ \tele_\afterstart$ as defined previously, and $ \tele_\start = 4\step^2\lips^2\norm{\state_\start - \state_\paststart}^2 \le \norm{\state_\start - \state_\paststart}^2$. Thus, Lemma~\ref{lem:descent} enables us to conclude the proof for \acf{PEG}.

%\textbf{\acf{OG}.}\enspace
\Paragraph{\acf{OG}\afterhead}
The update of \ac{OG} with constant step-size $\step$ can be written as
\begin{equation}
    \begin{cases}
    \inter = \proj_{\points}(\current - \step\vecfield(\past))\\
    \update = \current - (\current - \inter + \step\vecfield(\inter) - \step\vecfield(\past))
\end{cases}
\end{equation}
%where $\vecspace$ is the full space so that $\proj_{\vecspace} $ is the identity. 
%
In that form, we can use \autoref{lem:4points} \ref{lem:4points-a} with $(\point, \dpoint_1, \dpoint_2, \pointnew_1, \pointnew_2, \cvx_1, \cvx_2)
\subs(\current, \step\vecfield(\past), \current - \inter + \step\vecfield(\inter) - \step\vecfield(\past), \inter, \update, \points, \vecspace)$ to get
\begin{align}
  \norm{\update - \arpoint}^2
  &= 
  \norm{\current-\arpoint}^2
  + \norm{\update - \inter}^2 - \norm{\inter - \current}^2
  \notag\\
  & \enspace
  - 2\product{\current-\inter+\step\vecfield(\inter)-\step\vecfield(\past)}{\inter-\arpoint} .
    \label{eq:ogda_det_descent}
\end{align}

One the one hand, since $\inter = \proj_{\points}(\current - \step\vecfield(\past))$
and $\arpoint\in\points$, we have
\begin{equation}
    \product{\inter-(\current-\step\vecfield(\past))}{\inter-\arpoint} \le 0.
    \label{eq:ogda_det_proj}
\end{equation}
On the other other hand, by definition of $\update$ and the $\lips$-Lipschitz continuity of $\vecfield$,
\begin{equation}
\norm{\update-\inter}^2 = \step^2\norm{\vecfield(\inter) - \vecfield(\past)}^2 \le \step^2\lips^2\norm{\inter - \past}^2.
\label{eq:ogda_det_diff}
\end{equation}
Then, applying the same arguments used to get \eqref{eq:peg_det_trick}, we can show that
for all $\run\ge\afterstart$,
\begin{equation}
    \norm{\inter - \past}^2
    \le 4\norm{\inter - \current}^2
    + \norm{\past - \pastpast}^2 - \norm{\inter - \past}^2.
    \label{eq:ogda_det_tele}
\end{equation}
Putting together \eqref{eq:ogda_det_descent}, \eqref{eq:ogda_det_proj},
\eqref{eq:ogda_det_diff}, and \eqref{eq:ogda_det_tele}, we obtain for  $\step\le 1/(2\lips)$ and for all $t\ge\afterstart$, 
\begin{align}
  &\norm{\update-\arpoint}^2
  \notag\\
  &\enspace\le
  \norm{\current-\arpoint}^2
  -2\step\product{\vecfield(\inter)}{\inter-\arpoint}
  + \step^2\lips^2\norm{\inter - \past}^2
  - \norm{\inter - \current}^2
  \notag\\
  &\enspace\le
  \norm{\current-\arpoint}^2
  -2\step\product{\vecfield(\inter)}{\inter-\arpoint}
  + \step^2\lips^2(\norm{\past - \pastpast}^2 - \norm{\inter - \past}^2).
\end{align}
Finally, since \eqref{eq:peg_det_init} is still true using the same argument as for \ac{PEG}, \eqref{eq:descent} is satisfied by choosing the same $\seqinf{\tele}{\run}$
and $\seqinf{\scalar}{\run}$ as in the case of \ac{PEG}; the same result thus holds for \acf{OG}.

%\textbf{\acf{RG}.}\enspace
\Paragraph{\acf{RG}\afterhead}
We recall the update rule of \ac{RG}
\begin{equation}
\begin{cases}
    \inter = \current - (\last - \current)\\
    \update = \proj_{\points}(\current - \step\vecfield(\inter)).
\end{cases}
\end{equation}
As in the previous cases, we use \autoref{lem:4points}. Using the first inequality of Part~\ref{lem:4points-b} with $(\point, \dpoint_1, \dpoint_2, \pointnew_1, \pointnew_2, \cvx_1, \cvx_2)
\subs(\current, \last - \current, \step\vecfield(\inter), \inter, \update, \vecspace, \points)$, we get 
\begin{align}
    \norm{\update - \arpoint}^2
    &\le
    \norm{\current-\arpoint}^2
    + 2\product{\step\vecfield(\inter)-(\last-\current)}{\inter-\update}\\
    &\enspace
    - 2\step\product{\vecfield(\inter)}{\inter-\arpoint}
    - \norm{\update - \inter}^2 - \norm{\inter - \current}^2.
    \label{eq:rg_det_descent_prel}
\end{align}
%
% we substitute
% $\point\subs\current$, $\dpoint_1\subs\last - \current$, $\dpoint_2\subs\step\vecfield(\inter)$,
% $\cvx_1\subs\vecspace$ and $\cvx_2\subs\points$.
% In particular, we would like to use the first inequality of \autoref{lem:4points} \ref{lem:4points-b}.
% To simplify, we start by looking at the quantity
% $2\product{\step\vecfield(\inter)-(\last-\current)}{\inter-\update}$
% which comes from the $2\product{\dpoint_2-\dpoint_1}{\pointnew_1-\pointnew_2}$
% in \eqref{eq:4points-b}.
As $\current = \proj_{\points}(\last - \step\vecfield(\past))$ and $\last,\update\in\points$,
it follows
\begin{align}
    \product{\current-(\last-\step\vecfield(\past))}{\current-\last} & \le 0,\label{eq:rg_proj_1}\\
    \product{\current-(\last-\step\vecfield(\past))}{\current-\update} & \le 0.\label{eq:rg_proj_2}
\end{align}
By summing \eqref{eq:rg_proj_1} and \eqref{eq:rg_proj_2}
%, writing $\inter = 2\current - \last$, 
and rearranging the terms, we get
\begin{equation}
    \product{\current-\last}{\inter-\update} 
    \le -\product{\step\vecfield(\past)}{\inter-\update},
\end{equation}
thus,
\begin{align}
%\begin{multlined}[b][0.7\displaywidth]
    &2\product{\step\vecfield(\inter)-(\last-\current)}{\inter-\update}
    \notag\\
    &\enspace\le 2\product{\step\vecfield(\inter)-\step\vecfield(\past)}{\inter-\update}
    \notag\\
    &\enspace\le 2\step\lips\norm{\inter-\past}\norm{\inter-\update}.
    \label{eq:rg_proj_sum}
%\end{multlined}
\end{align}

Combining \eqref{eq:rg_det_descent_prel} and \eqref{eq:rg_proj_sum}, we get 
\begin{align}
    \norm{\update - \arpoint}^2
    &\le
    \norm{\current-\arpoint}^2
    + 2\step\lips\norm{\inter-\past}\norm{\inter-\update}
    \notag\\
    &\enspace
    - 2\step\product{\vecfield(\inter)}{\inter-\arpoint}
    - \norm{\update - \inter}^2 - \norm{\inter - \current}^2.
    \label{eq:rg_det_descent}
\end{align}
By using twice Young's inequality: i)
$2\product{a}{b}\le\youngeps\norm{a}^2+(1/\youngeps)\norm{b}^2$
with $\youngeps = 1/\sqrt{2}$;
then ii)  $\norm{a+b}^2\le(1+\alt\youngeps)\norm{a}^2+(1+1/\alt\youngeps)\norm{b}^2$
with $\alt\youngeps = 1+\sqrt{2}$, we have
\begin{align}
    &2\norm{\inter-\past}\norm{\inter-\update}
    \notag\\
    &\enspace\le \frac{1}{\sqrt{2}}\norm{\inter-\past}^2 + \sqrt{2}\norm{\inter-\update}^2
    \notag\\
    &\enspace\le (1+\sqrt{2})\norm{\inter-\current}^2
        + \norm{\current-\past}^2 + \sqrt{2}\norm{\inter-\update}^2.
    \label{eq:rg_det_young}
\end{align}
Substituting \eqref{eq:rg_det_young} into \eqref{eq:rg_det_descent} yields
\begin{align}
    \norm{\update - \arpoint}^2
    &\le
    \norm{\current-\arpoint}^2
    - 2\step\product{\vecfield(\inter)}{\inter-\arpoint}
    \notag\\
    &\enspace
    + ((1+\sqrt{2})\step\lips - 1)\norm{\inter - \current}^2
    \notag\\
    &\enspace
    + \step\lips\norm{\current-\past}^2
    - (1-\sqrt{2}\step\lips)\norm{\update - \inter}^2
    \notag\\
    &\le
    \norm{\current-\arpoint}^2
    - 2\step\product{\vecfield(\inter)}{\inter-\arpoint}
    \notag\\
    &\enspace
    + \step\lips\norm{\current-\past}^2 - \step\lips\norm{\update-\inter}^2,
    \label{eq:rg_det_tele}
\end{align}
where in the last line, we used twice that $\step\le1/((1+\sqrt{2})\lips)$.
Once again, \eqref{eq:descent} is verified with the choice
$\forall\run\in\N, \tele_\run = \step\lips\norm{\current-\past}^2,
\scalar_\run = \step$ and the result thus holds for \acf{RG}.
We also notice that $\tele_\start\le\norm{\state_\start-\state_\paststart}^2$ since $\step\lips < 1$.
%\end{proof}

% =================================================

\subsection{\cref{lem:descent} with other suboptimality measures}

Here we discuss how the statement of \cref{lem:descent}, and consequently also that of \cref{thm:det-global-erg}, can be adjusted
to consider more adapted convergence measures in the cases of loss minimization and min-max optimization.
The notations are those of \cref{ex:function,ex:saddle}, and we write
$\avg[\point] = (\sum_{\runalt=\start}^{\nRunsNew} \current[\scalar][\runalt])^{-1}\sum_{\runalt=\start}^{\nRunsNew} \current[\scalar][\runalt] \inter[\state][\runalt]$.

\Paragraph{Loss minimization\afterhead}
$\vecfield=\nabla\obj$ is monotone implies the convexity of $\obj$, so
\begin{align}
  \product{\vecfield(\inter[\state][\runalt])}{\inter[\state][\runalt]-\arpoint}
  = \product{\nabla\obj(\inter[\state][\runalt])}{\inter[\state][\runalt]-\arpoint}
  \ge \obj(\inter[\state][\runalt]) - \obj(\arpoint).
\end{align}
With Jensen's inequality we get,
\begin{equation}
    \left(\sum_{\runalt=\start}^{\run}\scalar_\runalt\right)^{-1}
    \sum_{\runalt=\start}^{\run}
    \scalar_\runalt\product{\vecfield(\inter[\state][\runalt])}{\inter[\state][\runalt]-\arpoint}
    \ge
    \left(\sum_{\runalt=\start}^{\run}\scalar_\runalt\right)^{-1}
    \sum_{\runalt=\start}^{\run}
    \scalar_\runalt\obj(\inter[\state][\runalt]) - \obj(\arpoint)
    \ge
    \obj(\avg[\point]) - \obj(\arpoint)
\end{equation}
This is true for any $\arpoint\in\points$, and especially for $\arpoint\in\sols$.
Let $\radius=\dist(\point_\start, \sols)$. By invoking \eqref{eq:descent_tele}, we conclude
\begin{equation}
    \obj(\avg[\point])-\min\obj \le \frac{\radius^{2} + \tele_\start}{2 \sum_{\runalt=\start}^{\nRunsNew} \current[\scalar][\runalt]}.
\end{equation}

\Paragraph{Min-max optimization\afterhead}
$\vecfield=(\nabla_{\minvar}\sadobj,-\nabla_{\maxvar}\sadobj)$ being monotone is equivalent to $\sadobj$ being convex-concave.
In such saddle-point problems, the quality of a candidate solution $\test = (\test[\minvar],\test[\maxvar])$
is often assessed via the \emph{\acl{NI}} function \citep{NK55}, defined here as
\begin{equation}
\label{eq:NI}
\tag{NI}
\nikiso(\test)
	= \sup_{\maxvar\in\maxvars} \sadobj(\test[\minvar],\maxvar)
	- \inf_{\minvar\in\minvars} \sadobj(\minvar,\test[\maxvar])
\end{equation}
provided of course that the \acl{RHS} is well-posed.
Its restricted variant $\resnikiso$ can also be defined by analogy with the definition of $\reserr$.

% %
% \begin{lemma}
% \label{lem:err_bound}
% Suppose either $\vecfield$ is monotone in the general case or the objective
% function is convex-concave for the saddle point problem.
% Let $\seq{\point}{\startg}{\finishg}\in\vecspace$,
% $\seq{\scalar}{\startg}{\finishg}\in\R^+$ such that
% $\sum_{\indg=\startg}^{\finishg}\scalar_{\indg}=1$.
% We denote by
% $\avg[\point] \defeq \sum_{\indg=\startg}^{\finishg}\scalar_{\indg}\point_{\indg}$.
% Then
% \begin{equation}
% %\label{eq:err_bound}
%   \reserr(\avg[\point])
%   \le
%   \max_{\substack{\arpoint\in\points
%   \notag\\ \norm{\arpoint-\refpoint}\le\radius}}
%   \sum_{\indg=\startg}^{\finishg}
%   \scalar_{\indg}\product{\vecfield(\point_{\indg})}{\point_{\indg}-\arpoint}.
% \end{equation}
% \end{lemma}

% \begin{proof}
Let us denote $\inter[\state][\runalt] = (\inter[\minvar][\runalt], \inter[\maxvar][\runalt])$ and
$\arpoint = (\minvar, \maxvar)$. By convex-concavity of $\sadobj$, it holds
\begin{align}
  \product{\vecfield(\inter[\state][\runalt])}{\inter[\state][\runalt]-\arpoint}
  &
  = \product{\nabla_{\minvar}\sadobj(\inter[\minvar][\runalt], \inter[\maxvar][\runalt])}{\inter[\minvar][\runalt]-\minvar}
  - \product{\nabla_{\maxvar}\sadobj(\inter[\minvar][\runalt], \inter[\maxvar][\runalt])}{\inter[\maxvar][\runalt]-\maxvar}
  \notag\\
  & \ge
  \sadobj(\inter[\minvar][\runalt], \inter[\maxvar][\runalt]) - \sadobj(\minvar, \inter[\maxvar][\runalt])
  + \sadobj(\inter[\minvar][\runalt], \maxvar) - \sadobj(\inter[\minvar][\runalt], \inter[\maxvar][\runalt])
  \notag\\
  & = \sadobj(\inter[\minvar][\runalt], \maxvar) - \sadobj(\minvar, \inter[\maxvar][\runalt]).
\end{align}
We can again apply Jensen's inequality to show that
\begin{equation}
\label{eq:err_bound_s}
\left(\sum_{\runalt=\start}^{\run}\scalar_\runalt\right)^{-1}
    \sum_{\runalt=\start}^{\run}
    \scalar_\runalt\product{\vecfield(\inter[\state][\runalt])}{\inter[\state][\runalt]-\arpoint}
    \ge
    \sadobj(\avg[\minvar], \maxvar) - \sadobj(\minvar, \avg[\maxvar]),
\end{equation}
where we write $\avg[\point] = (\avg[\minvar], \avg[\maxvar])$.
By \eqref{eq:descent_tele} and definition of the \acl{NI} function, maximizing over $(\minvar,\maxvar)\in\points\intersect\ballr{\state_\start}{\radius}$ gives
\begin{equation}
     \resnikiso(\avg[\point]) \le \frac{\radius^{2} + \tele_\start}{2 \sum_{\runalt=\start}^{\nRunsNew} \current[\scalar][\runalt]}.
\end{equation}

% =================================================
\subsection{Proof of \cref{thm:det-local}}
\label{app:det-local}

Here we provide a quick proof of \cref{thm:det-local}. % based on \eqref{eq:descent} and \cref{thm:det-global-last}.
We do not try to optimize the constants and better results could be derived by examining each algorithm carefully.
Note that since $\ac{RG}$ can evaluate $\vecfield$ at infeasible points, we need to strengthen condition \eqref{eq:Jac} in \cref{def:regular} to consider all $\tvec$ in the \emph{tangent span} of $\points$, \ie the subspace of $\vecspace$ spanned by all possible displacement vectors of the form $\tvec = \pointalt - \point$, $\point,\pointalt\in\points$.
%\PM{Changed things a bit here.}

%First of all, let us notice that
In order to show a local geometric convergence rate we only need to show that by choosing sufficiently small constant step-size and initializing at points
sufficiently close to $\sol$, we ensure $\current\in\cpt$ for all $\run\in\N/2$ (where $\cpt$ is defined in \cref{lem:regualr} and this is in view of \cref{thm:det-global-last}).  In fact, although \cref{thm:det-global-last} is stated for strongly monotone operators, by carefully examining its
proof, it turns out that we only need
$\product{\vecfield(\inter)}{\inter-\sol}\ge\strong\norm{\inter-\sol}^2$ for some constant $\strong>0$ and all $\run\in\N$.
%this assumption can be weakend to $\product{\vecfield(\point)}{\point-\sol}\ge\strong\norm{\point-\sol}^2$ for some
%constant $\strong>0$ and all $\point\in\cvx$
%where $\cvx = \points$ for \ac{PEG} and \ac{OG} and $\cvx=\vecspace$ for \ac{RG}.
% We refer the reader to \cite{Mal15,GBVV+19,MOP19} for more details. 
We thus proceed to show that $\forall \run\in\N/2, \current\in\cpt$. To do so, let us show that one can choose the initial points and $\step$ so that $\forall \run\in\N$, (\emph{i}) $\norm{\current-\sol}^2 \le \frac{\nhdradius^2}{4}$; (\emph{ii}) $\inter\in\cpt$.

\underline{Part (\emph{i}).} It is proved in \cref{app:det-ergodic-proof} that the iterates of the \ac{SEG} methods verify \eqref{eq:descent} under
\cref{asm:Lipschitz} (Lipschitz continuity) if $\step$ is smaller than some constant.
By \cref{lem:regualr} we know that $\vecfield$ is indeed Lipschitz continuous on the compact $\cpt$.
Suppose that for all $\runalt\in\N/2$, $\runalt\le\run$, we have $\current[\state][\runalt]\in\cpt$, then
it holds $\product{\vecfield(\inter[\state][\runalt])}{\inter[\state][\runalt]-\sol}\ge0$ for all $\runalt\in\intinterval{\start}{\run-1}$.
This is true for \ac{PEG} and \ac{OG} because $\inter[\state][\runalt]\in\points$ and subsequently $\inter[\state][\runalt]\in\nhd=\points\union\cpt$. %\FI{What is $\nhd$ here? Is $\cpt = \ball_{r}(\sol) \cap \points$  or is it $\nhd$ ?}
%\YGH{$\cvt$ is the ball and $\nhd$ is the intersection of $\points$ and $\cpt$.}
For \ac{RG} we did mention above that we need to relax the definition of a regular solution to consider all the $\tvec\in\vecspace$ and the statement of \cref{lem:regualr} can also be modified accordingly.
Using \eqref{eq:descent}, we obtain\footnote{Please refer to the proof of \cref{thm:det-global-erg} for the exact value of $\current[\tele]$.}%by the proof of \cref{thm:det-global-erg} that
\begin{equation}
    \norm{\current-\sol}^2 + \current[\tele] \le \norm{\state_\start-\sol}^2 + \tele_\start.
\end{equation}
 for the three algorithms with  $\current[\tele]\ge 0$. %   with $\tele_\start\le\norm{\state_\start-\state_\paststart}^2$.
By imposing $\state_\paststart=\state_\start$ in \ac{PEG} and \ac{OG}, we get $\tele_\start = 0$.
Similarly, we may impose $\state_\laststart=\state_\paststart$ in \ac{RG}, leading to
$\tele_\start \le \norm{\state_\start-\state_\laststart}^2 \le \step^2 \norm{\vecfield(\state_\laststart)}^2$.
It is thus possible to choose the adequate initial points and $\step$ such that
$\norm{\state_\start-\sol}^2 + \tele_\start \le \frac{\nhdradius^2}{4}$, which in turn guarantees
$\norm{\current-\sol}^2 \le \frac{\nhdradius^2}{4}$. %$\forall \run\in\N$.

\underline{Part (\emph{ii}).}  We now proceed to prove that we may choose $\step$ sufficiently small such that if $\norm{\current-\sol}^2\le \frac{\nhdradius^2}{4}$ and
$\past\in\cpt$ then $\inter\in\cpt$.
We notice that for the three algorithms, we have
\begin{equation}
    \norm{\inter-\current}^2 \le \step^2\norm{\vecfield(\past)}^2
\end{equation}
by the non-expansiveness of the projection.%
\footnote{In particular this also holds for \ac{RG} since then
$\inter-\current=\current-\last=\proj_{\points}(\last - \step\vecfield(\past))-\proj_{\points}(\last)$.}
We define $\vbound\defeq\sup_{\point\in\cpt}\norm{\vecfield(\point)} < \infty$ where the finiteness of $\vbound$ comes from the continuity of $\vecfield$ and the boundedness of $\cpt$.
We choose $\gamma\le\nhdradius/(2\vbound)$ so that $\step^2\norm{\vecfield(\past)}^2\le\frac{\nhdradius^2}{4}$ since $\past\in\cpt$. Then, by Young's inequality, we get
\begin{equation}
    \norm{\inter-\sol}^2 \le 2\norm{\inter-\current}^2 + 2\norm{\current-\sol}^2 \le \nhdradius^2.    
\end{equation}
%
%Finally, taking adequate initial points and a stepsize $\step$ such that
%$\norm{\state_\start-\sol}^2 + \tele_\start \le \frac{\nhdradius^2}{4}$ and  $\gamma\le\nhdradius/(2\vbound)$, we get $\inter\in\cpt$.
In other words, $\inter\in\cpt$.

\underline{Conclusion.} We first notice that the conditions on the initial points and the stepsize $\step$ do not depend on the iteration. Thus, by simple induction we have that if we initialize the algorithm such that
$$\step\le\nhdradius/(2\vbound) \quad\text{ and }\quad \norm{\state_\start-\sol}^2 + \tele_\start \le \frac{\nhdradius^2}{4},$$
then for all $\run\in\N/2, \current\in\cpt$, concluding the proof.

%% file: App-StochProofs.tex
%----------------------------------------------------------------------
%%% APP: STOCHASTIC PROOFS
%----------------------------------------------------------------------
% !TEX root = ./Main.tex

Let us focus in this section on the \eqref{eq:PEG} algorithm:
\begin{equation}
\tag{PEG}
\begin{aligned}
\inter
	&= \proj_{\points}(\current - \current[\step]\past[\vecfield])
	\\
\update
	&= \proj_{\points}(\current - \current[\step]\inter[\vecfield])
\end{aligned}
\end{equation}

Following \cref{app:det-ergodic-proof}, we initialize the
algorithm with random $\state_\paststart$ and $\state_\start$ in $\points$.
Recall that % $\samples$ the underlying sample space,
$\seqinf[\frac{\N}{2}]{\filter}{\run}$ denotes the natural filtration
associated with the sequence $\seqinf[\frac{\N}{2}]{\state}{\run}$. %[\N\union\N+\frac{1}{2}]
In the \ac{PEG} algorithm, we have $\current[\filter] = \inter[\filter]$ for all $\run\in\N$ (thus $\inter$ is $\filter_\run$-measurable) so the zero-mean hypothesis \eqref{eq:mean} can be written as $\exof{\inter[\noise]\given{\current[\filter]}} = 0$.
%The noise term $\inter[\noise] \defeq \inter[\vecfield]-\vecfield(\inter)$ is independent of $\filter_\run$ of and verifies $\exof{\inter[\noise]\given{\current[\filter]}} = 0$.
% \FI{I don't think we need the conditioning here as $\inter[\noise]$ is independent of $\filter_\run$.}
% \YGH{They are not independent because in principle $\inter[\noise]$ depend on $\inter$. We can of course write $\ex[\inter[\noise]]=0$,
% but I suppose the conditional form is better here because this is what is used later.}

\subsection{Proof of \cref{thm:stoch-global}}
\label{app:stoch-global-proofs}

%\ThStochGlobal*

%\begin{proof}
\paragraph{Last iterate convergence\afterhead}
As in the proof of \autoref{thm:det-global-erg}, we first apply \autoref{lem:4points} \ref{lem:4points-b} with $(\point, \dpoint_1, \dpoint_2, \pointnew_1, \pointnew_2, \cvx_1, \cvx_2)
\subs(\current, \current[\step]\past[\vecfield], \current[\step]\inter[\vecfield], \inter, \update, \points, \points)$ and the solution $\sol\in\points$ as a trial point to obtain
\begin{align}
  \norm{\update-\sol}^2
  & \le
  \norm{\current-\sol}^2
  - 2\current[\step]\product{\inter[\vecfield]}{\inter-\sol}
  \notag\\
  & \enspace
  + \current[\step]^2\norm{\inter[\vecfield] - \past[\vecfield]}^2
  - \norm{\inter - \current}^2.
  \label{eq:peg_stoch_4points}
\end{align}
%
%Since
%$\ex[\exof{\product{\inter[\noise]}{\vecfield(\inter) - \past[\vecfield]}\given{\current[\filter]}}]
%=\ex[\product{\exof{\inter[\noise]\given{\current[\filter]}}}{\vecfield(\inter) - \past[\vecfield]}]=0$.
%The expectation of $\norm{\inter[\vecfield] - \past[\vecfield]}^2$ can be decomposed as 
The following holds true thanks to the law of total expectation, 
\begin{align}
    \ex[ & \norm{\inter[\vecfield] - \past[\vecfield]}^2]
    \notag\\ 
     & =  
     \ex[\norm{\vecfield(\inter) - \past[\vecfield]}^2
    + 2\product{\inter[\noise]}{\vecfield(\inter) - \past[\vecfield]}
    + \norm{\inter[\noise]}^2]
    \notag\\
    & = \ex[\norm{\vecfield(\inter) - \past[\vecfield]}^2]
    + 2\ex[\exof{\product{\inter[\noise]}{\vecfield(\inter) - \past[\vecfield]}\given{\current[\filter]}}]
    + \ex[\norm{\inter[\noise]}^2]
    \notag\\
    & = \ex[\norm{\vecfield(\inter) - \past[\vecfield]}^2] + \ex[\norm{\inter[\noise]}^2].
    \label{eq:peg_stoch_noise_ex}
\end{align}
By Young's inequality, $\lips$-Lipschitz continuity of $\vecfield$,
and non-expansiveness of the projection, we have
\begin{align}
    & \norm{\vecfield(\inter) - \past[\vecfield]}^2
     \le
    2\norm{\vecfield(\inter) - \vecfield(\past)}^2 + 2\norm{\past[\noise]}^2
    \notag\\
    &~~~~~~~~~~~~~~~ \le
    2\lips^2\norm{\inter - \past}^2 + 2\norm{\past[\noise]}^2
    \notag\\
    &~~~~~~~~~~~~~~~ \le
    4\lips^2\norm{\inter-\current}^2
    + 4\lips^2\norm{\current-\past}^2
    + 2\norm{\past[\noise]}^2
    \notag\\
    &~~~~~~~~~~~~~~~ \le
    4\lips^2\norm{\inter-\current}^2
    + 4\last[\step]^2\lips^2\norm{\past[\vecfield]-\pastpast[\vecfield]}^2
    + 2\norm{\past[\noise]}^2.
    \label{eq:peg_stoch_noise_young}
\end{align}
Notice that the choice $\strongstepm\ge4\lips\step$ implies
$8\current[\step]^2\lips^2+2\current[\step]\lips\le1$, which in turn yields
$8\current[\step]^2\lips^2\le1-\strong\current[\step]$.
Combining \eqref{eq:peg_stoch_noise_ex} and \eqref{eq:peg_stoch_noise_young}, similarly to \eqref{eq:peg_det_trick}, we can thus show that
\begin{align}
    \ex[\norm{\inter[\vecfield] - \past[\vecfield]}^2]
    &\le
    8\lips^2\ex[\norm{\inter-\current}^2]
    + 8\last[\step]^2\lips^2\ex[\norm{\past[\vecfield]-\pastpast[\vecfield]}^2]
    \notag\\
    &~~~~~
    + 4\ex[\norm{\past[\noise]}^2] + 2\ex[\norm{\inter[\noise]}^2]
    -\ex[\norm{\inter[\vecfield] - \past[\vecfield]}^2]
    \notag\\
    &\le
    8\lips^2\ex[\norm{\inter-\current}^2]  + 6\noisevar
    \notag\\
    &~~~~~
    +     \frac{\last[\step]^2}{\current[\step]^2}
    (1-\strong\current[\step])\ex[\norm{\past[\vecfield]-\pastpast[\vecfield]}^2]
    -\ex[\norm{\inter[\vecfield] - \past[\vecfield]}^2],
    \label{eq:peg_stoch_tele_cal}
\end{align}
where in the last line we also use
$\ex[\norm{\past[\noise]}^2]\le\noisevar$, $\ex[\norm{\inter[\noise]}^2]\le\noisevar$.

We also have
\begin{equation}
    \ex[\product{\inter[\vecfield]}{\inter-\sol}]
    = \ex[\exof{\product{\inter[\vecfield]}{\inter-\sol}\given{\current[\filter]}}]
    = \ex[\product{\vecfield(\inter)}{\inter-\sol}].
    \label{eq:peg_stoch_mid_ex}
\end{equation}
Since $\sol$ is the unique solution of \eqref{eq:SVI}, it follows
$\product{\vecfield(\sol)}{\inter-\sol}\ge0$.
Consequently, with strong monotonicity of $\vecfield$, we get
\begin{equation}
    \product{\vecfield(\inter)}{\inter-\sol}
    \ge \product{\vecfield(\inter)-\vecfield(\sol)}{\inter-\sol}
    \ge \strong\norm{\inter-\sol}^2.
\end{equation}
By Young's inequality
\begin{align}
    \norm{\current-\sol}^2\le2\norm{\current-\inter}^2+2\norm{\inter-\sol}^2,
\end{align}
we can further write
\begin{equation}
    \product{\vecfield(\inter)}{\inter-\sol}
    \ge \frac{\strong}{2}\norm{\current-\sol}^2 - \strong\norm{\current-\inter}^2.
    \label{eq:peg_stoch_strong_bound}
\end{equation}
Taking expectation over \eqref{eq:peg_stoch_4points} and using
\eqref{eq:peg_stoch_tele_cal}, \eqref{eq:peg_stoch_mid_ex},
\eqref{eq:peg_stoch_strong_bound} leads to
\begin{align}
    \ex[\norm{\update-\sol}^2]
    & \le
    \ex[\norm{\current-\sol}^2]
    - \strong\current[\step]\ex[\norm{\current-\sol}^2]
    + 2\strong\current[\step]\ex[\norm{\current-\inter}^2]
    \notag\\
    & \enspace
    + \last[\step]^2
    (1-\strong\current[\step])\ex[\norm{\past[\vecfield]-\pastpast[\vecfield]}^2]
    \notag\\
    &\enspace
    + 8\current[\step]^2\lips^2\ex[\norm{\inter-\current}^2]
    - \current[\step]^2\ex[\norm{\inter[\vecfield] - \past[\vecfield]}^2]
    \notag\\
    &\enspace
    + 6\current[\step]^2\noisevar
    - \ex[\norm{\inter - \current}^2]
    \notag\\
    &=
    (1-\strong\current[\step])
    (\ex[\norm{\current-\sol}^2]
    + \last[\step]^2\ex[\norm{\past[\vecfield]-\pastpast[\vecfield]}^2])
    \notag\\
    &\enspace
    + 6\current[\step]^2\noisevar
    - \current[\step]^2\ex[\norm{\inter[\vecfield] - \past[\vecfield]}^2]
    \notag\\
    &\enspace
    + (8\current[\step]^2\lips^2 + 2\strong\current[\step] - 1)\ex[\norm{\inter - \current}^2].
    \label{eq:peg_stoch_tele_aux}
\end{align}
Using $8\current[\step]^2\lips^2 + 2\strong\current[\step] - 1\le0$,
\eqref{eq:peg_stoch_tele_aux} reduces to
\begin{equation}
\begin{multlined}[b][0.8\linewidth]
    \ex[\norm{\update-\sol}^2]
    + \current[\step]^2\ex[\norm{\inter[\vecfield] - \past[\vecfield]}^2]\\
    \le
    (1-\strong\current[\step])
    (\ex[\norm{\current-\sol}^2]
    + \last[\step]^2\ex[\norm{\past[\vecfield]-\pastpast[\vecfield]}^2])
    + 6\current[\step]^2\noisevar.
    \label{eq:peg_stoch_rec}
\end{multlined}
\end{equation}

We conclude by applying \autoref{lem:chung1954}
%\YGH{To be self-contained, maybe we can put the statement somewhere or even the proof (it does not seem to be so long)?}
with
$\seqitem_{\run}
\subs
\ex[\norm{\current-\sol}^2]
+ \last[\step]^2\ex[\norm{\past[\vecfield]-\pastpast[\vecfield]}^2]$,
$\consc\subs\strong\step$, %$p\subs1$, 
$\conscalt\subs6\step^2\noisevar$, and $\run_0\subs 2$,
which gives
\begin{equation}
    \ex[\norm{\state_\run-\sol}^2]
    + \step_{\run-1}^2\ex[\norm{\vecfield_{\run-\frac{1}{2}}-\vecfield_{\run-\frac{3}{2}}}^2]
    \le \frac{6\step^2\noisevar}{\strong\step-1}\frac{1}{\run} + \smalloh\left(\frac{1}{\run}\right).
\end{equation}
The second term on the \ac{LHS} of the inequality is always positive, and \eqref{eq:rate-stoch-global} follows immediately.
%\end{proof}

% T version
% \begin{equation}
%     \ex[\norm{\state_\nRuns-\sol}^2]
%     + \step_{\nRuns-1}^2\ex[\norm{\vecfield_{\nRuns-\frac{1}{2}}-\vecfield_{\nRuns-\frac{3}{2}}}^2]
%     \le \frac{6\step^2\noisevar}{\strong\step-1}\frac{1}{\nRuns} + \smalloh\left(\frac{1}{\nRuns}\right).
% \end{equation}

\paragraph{Ergodic convergence\afterhead}
The convergence of $\avg_\run$ as shown in \eqref{eq:rate-stoch-global-erg} can be deduce directly from above by using Jensen's inequality:
\begin{equation}
    \ex[\norm{\avg_\run-\sol}^2] \le \frac{1}{t}\sum_{\runalt=\start}^\run \ex[\norm{\current[\state][\runalt]-\sol}^2],
\end{equation}
and then we bound the \ac{RHS} of the inequality by \eqref{eq:rate-stoch-global}.

\subsection{Proof of \cref{thm:stoch-local}}
\label{app:stoch-local-proofs}
\input{App-StochProofs-Local}

%% file: App-StochProofs-Local.tex
%----------------------------------------------------------------------
%%% APP: STOCHASTIC LOCAL PROOF
%----------------------------------------------------------------------
% !TEX root = ./Main.tex

% I create a file for it because I think the proof will be quite long

%\ThStochLocal*

%\Paragraph{Notations\afterhead}

We start by defining some important quantities that will be used in our proof.
For any $\nRuns\ge1$, we set
\begin{align}
    \sumnoise_\nRuns
    &\defeq
    \sum_{\run=\start}^\nRuns
    2\step_\run\product{\inter[\noise]}{\inter-\sol},\\
    \sumnoisevar_\nRuns
    &\defeq
    \sum_{\run=\start}^\nRuns
    2\step_\run^2
    (\norm{\inter[\vecfieldstoch]}^2+\norm{\past[\vecfieldstoch]}^2),\\
    \sumnoiseall_\nRuns
    &\defeq
    \sumnoise_\nRuns^2 + \sumnoisevar_\nRuns.
\end{align}
Notice that $\sumnoise_\nRuns$, $\sumnoisevar_\nRuns$
and $\sumnoiseall_\nRuns$
are not $\filter_\nRuns$-measurable but
$\filter_{\nRuns+1}$-measurable (due to the terms in $\noise_{\nRuns+\frac{1}{2}}$ and $ \vecfieldstoch_{\nRuns+\frac{1}{2}}$).
For the sake of simplicity, we also write
$\inter[\snoise]\defeq\product{\inter[\noise]}{\inter-\sol}$ so that 
$\sumnoise_\nRuns = \sum_{\run=\start}^\nRuns 2\current[\step]\inter[\snoise]$
and $\exof{\inter[\snoise]\given{\current[\filter]}}=0$.

Regarding the choice of $\nhd$ and $\nhdalt$, we invoke \cref{lem:regualr} to obtain the corresponding $\strong$, $\nhdradius$ and $\nhd$.
We then set $\nhdalt \defeq \points\intersect\ballr{\sol}{\nhdradius/4}$. %and impose $\step>1/\strong$.
Let us consider the following events for $\nRuns\ge1$,
\begin{align}
    \eventalt_\nRuns
    &\defeq 
    \left\{\max_{\start\le\run\le\nRuns}\sumnoiseall_\run
    \le \noisebound \defeq \min\left(\frac{\nhdradius^2}{8}, \frac{\nhdradius^4}{16}\right)\right\},\\
    \event_\nRuns
    &\defeq
    \left\{\forall\run\in\intinterval{\start}{\nRuns}, \inter\in\nhd\right\}.
\end{align}
We additionally define
$\sumnoiseall_\laststart
\defeq 2\step_\start^2\norm{\vecfieldstoch_\paststart}^2$,
$\eventalt_\laststart \defeq \{\sumnoiseall_\laststart\le\noisebound\}$
and $\eventalt_\lastlaststart \defeq \event_\laststart \defeq \samples$,
where $\samples$ denotes the whole sample space.
It follows from the definitions that both
$(\eventalt_\nRuns)_{\nRuns\ge\lastlaststart}$ and
$(\event_\nRuns)_{\nRuns\ge\laststart}$ are
decreasing sequences of events.
Moreover, we have $\current[\eventalt][\nRuns]\in\update[\filter][\nRuns]$ while
$\current[\event][\nRuns]\in\current[\filter][\nRuns]$.
Also notice that $\event_\infty=\bigcap_{\nRuns\ge\laststart} \event_\nRuns$.
%we set $\event_\infty \defeq \bigcap_{\nRuns\ge\laststart} \event_\nRuns =\{\forall\run\ge\start, \inter\in\nhd\}$.

%\FI{Notations to add somewhere}
In terms of notation, for an event $E\subseteq \samples$, we denote by $\one_E$ its indicator function and $E^c$ its complementary. For any pair of events $E,F\subseteq \samples$, we denote by $\setexclude{E}{F}$ the event ``$E$ and not $F$'' \ie $E\cap F^c$.

%and $\vbound\defeq\sup_{\point\in\nhd}\vecfield(\point) < \infty$ where the finiteness of $\vbound$ comes from the Lipschitz continuity of $\vecfield$ (Assumption~\ref{asm:Lipschitz}) and the boundedness of $\nhd$ (by definition).

%\Paragraph{Lemmas\afterhead}
The proof of the theorem relies on the two following lemmas.

\begin{lemma}
\label{lem:event-inclusion}
For any $\nRuns\ge0$, we have the inclusion  $\last[\eventalt][\nRuns]\subseteq\event_\nRuns$.
\end{lemma}

\begin{proof} We prove this result by induction.

\underline{Initialization:}
$\eventalt_\lastlaststart\subseteq\event_\laststart$ is clear.
To prove that we also have $\eventalt_\laststart\subseteq\event_\start$, we use Young's inequality to get
\begin{equation}
    \norm{\state_\interstart-\sol}^2
    \le
    2\norm{\state_\interstart-\state_\start}^2
    + 2\norm{\state_\start-\sol}^2.
    \label{eq:local_stoch_inclu_init_young}
\end{equation}
On the one hand, since $\state_\start\in\nhdalt$ by assumption,
it holds $\norm{\state_\start-\sol}^2\le\frac{\nhdradius^2}{16}$.
On the other hand,
%\PM{Please check the last equality, I corrected a macro (sumoiseall) to make the files compile.}
%
\begin{equation}
    2\norm{\state_\interstart-\state_\start}^2
    = 2\norm{
        \proj_\points(\state_\start-\step_\start\vecfieldstoch_\paststart)
        -\proj_\points(\state_\start)}^2
    \le 2\step_\start\norm{\vecfieldstoch_\paststart}^2
    = \sumnoiseall_\laststart
    %\le \frac{\nhdradius^2}{8}.
\end{equation}

For any realization in $\eventalt_\laststart$, we have $2\step_\start\norm{\vecfieldstoch_\paststart}^2
    \le \frac{\nhdradius^2}{8}$; and so we can deduce from \eqref{eq:local_stoch_inclu_init_young} that $\norm{\state_\interstart-\sol}^2\le\frac{\nhdradius^2}{4}<\nhdradius^2$.
Since $\state_\interstart\in\points$,
it follows that $\state_\interstart\in\nhd$. This means that  $\eventalt_\laststart\subseteq\event_\start$.

\underline{Inductive step:}
Suppose that $\last[\eventalt][\nRuns]\subseteq\current[\event][\nRuns]$ holds
for some $\nRuns\ge\start$.
We would like to prove
$\current[\eventalt][\nRuns]\subseteq\update[\event][\nRuns]$.
To do so, we show that $\norm{\update[\state][\nRuns]-\sol}^2\le\frac{7}{16}\nhdradius^2$ for any realization in $\current[\eventalt][\nRuns]$. Applying \autoref{lem:4points} \ref{lem:4points-b} as in \eqref{eq:peg_stoch_4points}  yields for all
$\run\in\intinterval{\start}{\nRuns}$,
\begin{align}
    \norm{\update-\sol}^2
    & \le
    \norm{\current-\sol}^2
    - 2\current[\step]\product{\inter[\vecfieldstoch]}{\inter-\sol}
    \notag\\
    & \enspace
    + \current[\step]^2\norm{\inter[\vecfieldstoch] - \past[\vecfieldstoch]}^2
    - \norm{\inter - \current}^2
    \notag\\
    & \le
    \norm{\current-\sol}^2
    - 2\current[\step]\product{\vecfield(\inter)}{\inter-\sol}
    \notag\\
    & \enspace
    - 2\current[\step]\product{\inter[\noise]}{\inter-\sol}
    + 2\current[\step]^2
        (\norm{\inter[\vecfieldstoch]}^2+\norm{\past[\vecfieldstoch]}^2)
    \notag\\
    & \le
    \norm{\current-\sol}^2
    - 2\current[\step]\inter[\snoise]
    + 2\current[\step]^2
    (\norm{\inter[\vecfieldstoch]}^2+\norm{\past[\vecfieldstoch]}^2),
    \label{eq:local_stoch_inclu_4points}
\end{align}
where in the last line we can use
$\product{\vecfield(\inter)}{\inter-\sol}\ge0$
since by induction hypothesis,
%$\current[\event][\nRuns]$ also occurs
$\current[\eventalt][\nRuns]\subseteq\last[\eventalt][\nRuns]\subseteq\current[\event][\nRuns]$,
which means for any realization in $\current[\eventalt][\nRuns]$,  $\inter\in\nhd$  for all $\run\in\intinterval{\start}{\nRuns}$.

Summing \eqref{eq:local_stoch_inclu_4points} from $\run=\start$ to $\nRuns$
gives
\begin{align}
    \norm{\update[\state][\nRuns]-\sol}^2
    & \le
    \norm{\state_\start-\sol}^2
    - \sum_{\run=\start}^{\nRuns}2\current[\step]\inter[\snoise]
    + \sum_{\run=\start}^{\nRuns}
    2\current[\step]^2
    (\norm{\inter[\vecfieldstoch]}^2+\norm{\past[\vecfieldstoch]}^2)
    \notag\\
    & =
    \norm{\state_\start-\sol}^2
    - \current[\sumnoise][\nRuns]
    + \current[\sumnoisevar][\nRuns].
\end{align}
By definition of $\current[\eventalt][\nRuns]$, we have
$\current[\sumnoise][\nRuns]^2
\le\current[\sumnoiseall][\nRuns]\le\frac{\nhdradius^4}{16}$ (so $\abs{\current[\sumnoise][\nRuns]}\le\frac{\nhdradius^2}{4}$)
and
$\current[\sumnoisevar][\nRuns]
\le\current[\sumnoiseall][\nRuns]\le\frac{\nhdradius^2}{8}$.
Using that
$\norm{\state_\start-\sol}^2\le\frac{\nhdradius^2}{16}$ by assumption,
it follows immediately that 
$\norm{\update[\state][\nRuns]-\sol}^2\le\frac{7}{16}\nhdradius^2$.

Finally, in order to bound $\norm{\future[\state][\nRuns]-\sol}^2$, we again rely on Young's inequality:
\begin{align}
    \norm{\future[\state][\nRuns]-\sol}^2
    & \le 
    2 \norm{\future[\state][\nRuns]-\update[\state][\nRuns]}^2
    + 2 \norm{\update[\state][\nRuns]-\sol}^2
    \notag\\
    & \le
    2 \update[\step][\nRuns]^2\norm{\inter[\vecfieldstoch][\nRuns]}^2
    + 2 \norm{\update[\state][\nRuns]-\sol}^2.
    \label{eq:local_stoch_inclu_induct_young}
\end{align}

For any realization in $\current[\eventalt][\nRuns]$, we have that
\begin{align}
  &\text{i) } ~~~  2\update[\step][\nRuns]^2\norm{\inter[\vecfieldstoch][\nRuns]}^2
    \le2\current[\step][\nRuns]^2\norm{\inter[\vecfieldstoch][\nRuns]}^2
    \le\current[\sumnoisevar][\nRuns]
    \le\current[\sumnoiseall][\nRuns]\le\frac{\nhdradius^2}{8};\\
   &\text{ii) } ~~~    2 \norm{\update[\state][\nRuns]-\sol}^2 \le \frac{7}{8}\nhdradius^2.
\end{align}
Thus, \eqref{eq:local_stoch_inclu_induct_young} implies that 
$\norm{\future[\state][\nRuns]-\sol}^2\le\nhdradius^2$,
and subsequently $\future[\state][\nRuns]\in\nhd$.
As $\current[\eventalt][\nRuns]\subseteq\current[\event][\nRuns]$
and $\update[\event][\nRuns]=
\{\future[\state][\nRuns]\in\nhd\}\intersect\current[\event][\nRuns]$, we have proven that $\current[\eventalt][\nRuns]\subseteq\update[\event][\nRuns]$.
\end{proof}

\begin{lemma}
\label{lem:local-exp-prob-rec}
For $\run\ge\start$, we have the following recurrence inequality
\begin{equation}
    \ex[\current[\sumnoiseall][\run]\one_{\last[\eventalt][\run]}]
    \le
 \ex[\last[\sumnoiseall][\run]\one_{\lastlast[\eventalt][\run]}]
 +    \current[\step][\run]^2\lprbound
    - \noisebound\prob(\setexclude{\lastlast[\eventalt][\run]}{\last[\eventalt][\run]}),
    \label{eq:local_stoch_prob_tele}
\end{equation}
%\FI{Maybe the $\noisebound$ can be misinterpreted here as for any $\epsilon$ so I added its definition in the Lemma.}
%
where $\lprbound\defeq4\vbound^2+4\noisevar+4\nhdradius^2\noisevar$ and $\noisebound \defeq \min\left(\frac{\nhdradius^2}{8}, \frac{\nhdradius^4}{16}\right)$.

Moreover, if $\run=\start$, the bound can be refined to
\begin{equation}
    \ex[\sumnoiseall_\start \one_{\eventalt_\laststart}]
    \le
    \ex[\sumnoiseall_\laststart \one_{\eventalt_\lastlaststart}]
    +     \step_\start^2(2\vbound^2+2\noisevar+4\nhdradius^2\noisevar)
    - \noisebound\prob(\setexclude{\eventalt_\lastlaststart}{\eventalt_\laststart}).
    \label{eq:local_stoch_prob_tele_t1}
\end{equation}
\end{lemma}

\begin{proof}
We decompose
\begin{align}
    \ex[\current[\sumnoiseall][\run]\one_{\last[\eventalt][\run]}]
    & =
    \ex[(\current[\sumnoiseall][\run]-\last[\sumnoiseall][\run])\one_{\last[\eventalt][\run]}]
    + \ex[\last[\sumnoiseall][\run]\one_{\last[\eventalt][\run]}]
    \notag\\
    & =
    \ex[(\current[\sumnoiseall][\run]-\last[\sumnoiseall][\run])\one_{\last[\eventalt][\run]}]
    + \ex[\last[\sumnoiseall][\run]\one_{\lastlast[\eventalt][\run]}]
    - \ex[\last[\sumnoiseall][\run]\one_{\setexclude{\lastlast[\eventalt][\run]}{\last[\eventalt][\run]}}],
    \label{eq:local_stoch_prob_decompose}
\end{align}
where the second equality comes from the fact that as $\last[\eventalt][\run] \subseteq \lastlast[\eventalt][\run] $, we have $\last[\eventalt][\run] = \setexclude{\lastlast[\eventalt][\run]}{(\setexclude{\lastlast[\eventalt][\run]}{\last[\eventalt][\run]})} $.

For $\run\ge\afterstart$, we write
\begin{align}
    \current[\sumnoiseall][\run]
    &= 
    \current[\sumnoise][\run]^2 + \current[\sumnoisevar][\run]
    \notag\\
    &=
    \last[\sumnoise][\run]^2
    + 4 \current[\step][\run] \inter[\snoise][\run] \last[\sumnoise][\run] 
    + 4 \current[\step][\run]^2 \inter[\snoise][\run]^2
    + \last[\sumnoisevar][\run]
    + 2 \current[\step][\run]^2
        (\norm{\inter[\vecfieldstoch][\run]}^2+\norm{\past[\vecfieldstoch][\run]}^2)
    \notag\\
    &=
    \last[\sumnoiseall][\run]
    + 4 \current[\step][\run] \inter[\snoise][\run] \last[\sumnoise][\run] 
    + 4 \current[\step][\run]^2 \inter[\snoise][\run]^2
    + 2 \current[\step][\run]^2
        (\norm{\inter[\vecfieldstoch][\run]}^2+\norm{\past[\vecfieldstoch][\run]}^2).
    \label{eq:local_stoch_proba_QT_decompose}
\end{align}
Since $\last[\sumnoise][\run]$ and $\last[\eventalt][\run]$ are $\current[\filter][\run]$-measurable, we get
\begin{equation}
    \ex[ %4 \current[\step][\run]
        \inter[\snoise][\run]
        \last[\sumnoise][\run] \one_{\last[\eventalt][\run]}]
    =  %4 \current[\step][\run]
    \ex[
        \exof{\inter[\snoise][\run]\given{\current[\filter][\run]}}
        \last[\sumnoise][\run]  \one_{\last[\eventalt][\run]}
    ]
    = 0.
    \label{eq:local_stoch_proba_xi0}
\end{equation}
By \autoref{lem:event-inclusion},
$\last[\eventalt][\run]\subseteq\current[\event][\run]$ which means that for any realization in $\last[\eventalt][\run]$, 
we have $\inter[\state][\run]\in\nhd$. Therefore, $\norm{\inter[\state][\run]-\sol}^2\one_{\last[\eventalt][\run]} \le \nhdradius^2 \one_{\last[\eventalt][\run]}$ and consequently
\begin{align}
    \inter[\snoise][\run]^2\one_{\last[\eventalt][\run]}
    &= \product{\inter[\noise][\run]}{\inter[\state][\run]-\sol}^2
        \one_{\last[\eventalt][\run]}
        \notag\\
       &\le \norm{\inter[\noise][\run]}^2
        \norm{\inter[\state][\run]-\sol}^2\one_{\last[\eventalt][\run]} \le \norm{\inter[\noise][\run]}^2
        \nhdradius^2 \one_{\last[\eventalt][\run]}.
\end{align}
Using again that $\last[\eventalt][\run]$ is $\current[\filter][\run]$-measurable along with the boundedness of the variance of $\inter[\noise][\run]$ (see Eq.~\eqref{eq:variance}), we get
%\YGH{We may recall the assumption of bounded variance on $\nhd$. For the time being we do not have it in the text.}
%\PM{Is what I put sufficient for your needs?}
%\YGH{We can put specific hypothesis for the local case, to decide later.}
%
\begin{align}
    \ex[\inter[\snoise][\run]^2 \one_{\last[\eventalt][\run]}]
    &\le
    \nhdradius^2 \ex[\norm{\inter[\noise][\run]}^2 \one_{\last[\eventalt][\run]}] =
    \nhdradius^2
    \ex[
        \exof{\norm{\inter[\noise][\run]}^2\given{\current[\filter][\run]}}
        \one_{\last[\eventalt][\run]}
    ]
    \notag\\
    &\le
    \nhdradius^2
    \ex[\noisevar\one_{\last[\eventalt][\run]}] 
    =  \nhdradius^2\noisevar
    \prob[{\last[\eventalt][\run]}] 
    \le \nhdradius^2\noisevar.
    \label{eq:local_stoch_proba_xi2bound}
\end{align}
Applying once again the techniques above and relying on the boundedness of $\vecfield$ (as for any realization in  $ \last[\eventalt][\run] \subseteq \current[\event]$ we have $\inter[\state][\run]\in\nhd$ and $\vbound=\sup_{\point\in\nhd}\vecfield(\point) < \infty$), %, see the preamble),
we get
\begin{align}
    \ex[\norm{\inter[\vecfieldstoch][\run]}^2\one_{\last[\eventalt][\run]}]
    & =
    \ex[
        (\norm{\vecfield(\inter[\state][\run])}^2
        +2\product{\inter[\noise][\run]}{\vecfield(\inter[\state][\run])}
        +\norm{\inter[\noise][\run]}^2)
        \one_{\last[\eventalt][\run]}
    ]
    \notag\\
    & =
    \ex[
        \norm{\vecfield(\inter[\state][\run])}^2
        \one_{\last[\eventalt][\run]}
    ]
    \notag\\
    &\enspace
    + 2\ex[
        \exof{\product{\inter[\noise][\run]}{\vecfield(\inter[\state][\run])}
              \given{\current[\filter][\run]}}
        \one_{\last[\eventalt][\run]}]
    + \ex[\norm{\inter[\noise][\run]}^2  \one_{\last[\eventalt][\run]}]
    \notag\\
    & =
    \ex[
        \norm{\vecfield(\inter[\state][\run])}^2
        \one_{\last[\eventalt][\run]}
    ]
    + 0
    + \ex[
        \exof{\norm{\inter[\noise][\run]}^2\given{\current[\filter][\run]}}
        \one_{\last[\eventalt][\run]}
    ]
    \notag\\
    & \le
    \vbound^2 + \noisevar.
    \label{eq:local_stoch_proba_v2bound}
\end{align}
Using that $\last[\eventalt][\run]\subseteq\lastlast[\eventalt][\run]$
and repeating the arguments leading to \eqref{eq:local_stoch_proba_v2bound},
we have
\begin{equation}
    \ex[\norm{\past[\vecfieldstoch][\run]}^2\one_{\last[\eventalt][\run]}]
    \le
    \ex[\norm{\past[\vecfieldstoch][\run]}^2\one_{\lastlast[\eventalt][\run]}]
    \le
    \vbound^2 + \noisevar.
    \label{eq:local_stoch_proba_v2bound_bis}
\end{equation}
Combining
\eqref{eq:local_stoch_proba_QT_decompose},
\eqref{eq:local_stoch_proba_xi0},
\eqref{eq:local_stoch_proba_xi2bound},
\eqref{eq:local_stoch_proba_v2bound}
and
\eqref{eq:local_stoch_proba_v2bound_bis},
we get
\begin{equation}
    \ex[
        (\current[\sumnoiseall][\run]-\last[\sumnoiseall][\run])
        \one_{\last[\eventalt][\run]}
    ]
    \le
    \current[\step][\run]^2(4\vbound^2+4\noisevar+4\nhdradius^2\noisevar)
    = \current[\step][\run]^2\lprbound.
    \label{eq:local_stoch_prob_diff_bound}
\end{equation}
For the last term on the \ac{RHS} of \eqref{eq:local_stoch_prob_decompose},
we get by definition that for any realization in $\setexclude{\lastlast[\eventalt][\run]}{\last[\eventalt][\run]}$, $\last[\sumnoiseall][\run]>\noisebound$ and thus 
\begin{equation}
    \ex[
        \last[\sumnoiseall][\run]
        \one_{\setexclude{\lastlast[\eventalt][\run]}{\last[\eventalt][\run]}}
    ]
    \ge
        \noisebound
        \ex[
        \one_{\setexclude{\lastlast[\eventalt][\run]}{\last[\eventalt][\run]}}
    ]
    =
    \noisebound
    \prob(\setexclude{\lastlast[\eventalt][\run]}{\last[\eventalt][\run]}).
    \label{eq:local_stoch_prob_basic_ineq_bis}
\end{equation}
Substituting
\eqref{eq:local_stoch_prob_diff_bound}
and
\eqref{eq:local_stoch_prob_basic_ineq_bis}
into
\eqref{eq:local_stoch_prob_decompose}
gives exactly
\eqref{eq:local_stoch_prob_tele}.

The case $\run=\start$ is proved similarly.
In fact,
\begin{equation}
    \sumnoiseall_\start - \sumnoiseall_\laststart
    = 4\step_\start^2\snoise_\interstart^2
    + 2\step_\start^2\norm{\vecfieldstoch_\interstart}^2.
\end{equation}
Consequently by using $\eventalt_\laststart\subseteq\event_\start$,
%\eqref{eq:local_stoch_prob_diff_bound} also holds for $\run=\start$.
we have
\begin{equation}
    \ex[
        (\sumnoiseall_\start - \sumnoiseall_\laststart)
        \one_{\eventalt_\laststart}
    ]
    \le
    \step_\start^2(2\vbound^2+2\noisevar+4\nhdradius^2\noisevar)
    \label{eq:local_stoch_prob_diff_bound_t1}.
\end{equation}
By definition
$\setexclude{\eventalt_\lastlaststart}{\eventalt_\laststart}
= \{\sumnoiseall_\laststart>\noisebound\}$,
which shows \eqref{eq:local_stoch_prob_basic_ineq_bis} is equally true
with $\run=\start$.
\eqref{eq:local_stoch_prob_tele_t1} can then be immediately deduced from \eqref{eq:local_stoch_prob_decompose}.
%
%\begin{equation}
%    \ex[\sumnoiseall_\start \one_{\eventalt_\laststart}]
%    \le
%    \step_\start^2(2\vbound^2+2\noisevar+4\nhdradius^2\noisevar)
%    + \ex[\sumnoiseall_\laststart \one_{\eventalt_\lastlaststart}]
%    - \noisebound\prob(\setexclude{\eventalt_\lastlaststart}{\eventalt_\laststart}).
%    \label{eq:local_stoch_prob_tele_t1}
%\end{equation}
%
\end{proof}

\begin{proof}[Proof of \autoref{thm:stoch-local}]
%$ $\newline
\mbox{}
\begin{enumerate}[wide, label=(\alph*)]
%\begin{proofpart}{By choosing $\strongstepm$ sufficiently large,
%    we have $\forall\nRuns\ge\lastlaststart, \prob(\current[\eventalt][\nRuns]) \ge 1-\smallproba$}
\item
We first show that
by choosing $\strongstepm$ sufficiently large,
we have $\prob(\current[\eventalt][\nRuns]) \ge 1-\smallproba$ for all $\nRuns\ge\lastlaststart$ (when $\nRuns=-1$,  $\eventalt_{\lastlaststart}=\samples$). To do so, we will work on the complementary event ${\current[\eventalt][\nRuns]}^c 
= {\last[\eventalt][\nRuns]}^c
\union (\setexclude{\last[\eventalt][\nRuns]}{\current[\eventalt][\nRuns]})$
%= \eventalt_\lastlaststart \union \left(\Union_{\laststart\le\run\le\nRuns} (\setexclude{\last[\eventalt][\run]}{\current[\eventalt][\run]}) \right) =  \Union_{\laststart\le\run\le\nRuns}
%(\setexclude{\last[\eventalt][\run]}{\current[\eventalt][\run]})$
%\YGH{comment it out since not sure about the utility of the two equalities}
and prove that $\prob(\compprob{\current[\eventalt][\nRuns]})\le\smallproba$.
We start by bounding
$\prob(\setexclude{\last[\eventalt][\nRuns]}{\current[\eventalt][\nRuns]})$,
\begin{align}
    \noisebound\prob(\setexclude{\last[\eventalt][\nRuns]}{\current[\eventalt][\nRuns]})
    &= \noisebound
        \prob(\{\current[\sumnoiseall][\nRuns]>\noisebound\}\intersect\last[\eventalt][\nRuns])
        \notag\\
    &= \ex[\noisebound
        \one_{\{\current[\sumnoiseall][\nRuns]>\noisebound\}\intersect\last[\eventalt][\nRuns]}]
        \notag\\
    &\le \ex[\current[\sumnoiseall][\nRuns]
        \one_{\{\current[\sumnoiseall][\nRuns]>\noisebound\}\intersect\last[\eventalt][\nRuns]}]
        \notag\\
    &\le \ex[\current[\sumnoiseall][\nRuns]\one_{\last[\eventalt][\nRuns]}].
    \label{eq:local_stoch_prob_basic_ineq}
\end{align}
The last line is true since $\current[\sumnoiseall][\nRuns]$ is a positive random variable.

We now use \cref{lem:local-exp-prob-rec} by summing \eqref{eq:local_stoch_prob_tele} from $\run=\afterstart$ to $\nRuns$
and \eqref{eq:local_stoch_prob_tele_t1} which leads to
\begin{align}
    \ex[\current[\sumnoiseall][\nRuns]\one_{\last[\eventalt][\nRuns]}]
    & \le
    \ex[\sumnoiseall_\start \one_{\eventalt_\laststart}]
    + \sum_{\run=\afterstart}^{\nRuns}\current[\step][\run]^2\lprbound
    - \sum_{\run=\afterstart}^{\nRuns}\noisebound\prob(\setexclude{\lastlast[\eventalt][\run]}{\last[\eventalt][\run]})
    \notag\\
    & \le
    \ex[\sumnoiseall_\laststart \one_{\eventalt_\lastlaststart}]
    + \step_\start^2(2\vbound^2+2\noisevar+4\nhdradius^2\noisevar)
    + \sum_{\run=\afterstart}^{\nRuns}\current[\step][\run]^2\lprbound
    - \sum_{\run=\start}^{\nRuns}\noisebound\prob(\setexclude{\lastlast[\eventalt][\run]}{\last[\eventalt][\run]})
    \notag\\
    & =
    \ex[\sumnoiseall_\laststart]
    + \step_\start^2(2\vbound^2+2\noisevar+4\nhdradius^2\noisevar)
    + \sum_{\run=\afterstart}^{\nRuns}\current[\step][\run]^2\lprbound
    - \noisebound \prob(\compprob{\last[\eventalt][\nRuns]}),
    \label{eq:local_stoch_prob_tele_sum}
\end{align}
where in the last line we use that
$\eventalt_\lastlaststart=\samples$
and
$\compprob{\last[\eventalt][\nRuns]}
=\setexclude{\eventalt_\lastlaststart}{\last[\eventalt][\nRuns]}
=\disjUnion_{\start\le\run\le\nRuns}
(\setexclude{\lastlast[\eventalt][\run]}{\last[\eventalt][\run]})$ (
with $\disjUnion$ denoting the disjoint union) to get that $\prob(\compprob{\last[\eventalt][\nRuns]}) = \sum_{\run=\start}^{\nRuns}\prob(\setexclude{\lastlast[\eventalt][\run]}{\last[\eventalt][\run]})$.
Since we initialize with $\state_\paststart\in\nhd$, we have
\begin{equation}
    \ex[\sumnoiseall_\laststart]
    = 2\step_\start^2\ex[\norm{\vecfieldstoch_\paststart}^2]
    \le 2\step_\start^2(\vbound^2 + \noisevar).
    \label{eq:local_stoch_prob_Q0_bound}
\end{equation}
We set
$\stepsqsum \defeq \sum_{\run=\start}^{\infty} \current[\step]^2 < \infty$.
Combining
\eqref{eq:local_stoch_prob_basic_ineq},
\eqref{eq:local_stoch_prob_tele_sum}
and
\eqref{eq:local_stoch_prob_Q0_bound},
we obtain
\begin{align}
    \prob(\compprob{\eventalt_\nRuns})
    & =
    \prob(\setexclude{\last[\eventalt][\nRuns]}{\current[\eventalt][\nRuns]})
    + \prob(\compprob{\last[\eventalt][\nRuns]})
    \notag\\
    & \le
    \frac{1}{\noisebound}
    \ex[\current[\sumnoiseall][\nRuns]\one_{\last[\eventalt][\nRuns]}]
    + \prob(\compprob{\last[\eventalt][\nRuns]})
    \notag\\
    & \le
    \frac{1}{\noisebound}
    \sum_{\run=\start}^{\nRuns}\current[\step][\run]^2\lprbound
    - \prob(\compprob{\last[\eventalt][\nRuns]}) + \prob(\compprob{\last[\eventalt][\nRuns]})
    \le
    \frac{\stepsqsum\lprbound}{\noisebound}.
\end{align}

As $\stepsqsum $ converges to $0$ when $ \strongstepm \to \infty$, for any $\smallproba>0$ one can choose  $\strongstepm$ sufficiently large so that $\stepsqsum \le \smallproba\noisebound/\lprbound$; we then have
$\prob(\compprob{\eventalt_\nRuns})\le\smallproba$,
or equivalently,
$\prob(\eventalt_\nRuns)\ge 1- \smallproba$
for all $\nRuns\ge\lastlaststart$.
%\end{proofpart}
%\begin{proofpart}{Conclude}

Since $\last[\eventalt][\nRuns] \subseteq \current[\event][\nRuns]$
from \autoref{lem:event-inclusion}, 
we know that by choosing $\strongstepm$ sufficiently large, 
we have
$\prob(\current[\event][\nRuns]) \ge \prob(\last[\eventalt][\nRuns]) \ge 1-\smallproba$
for all $\nRuns\ge\laststart$.
As $(\event_\nRuns)_{\nRuns\ge\start}$ is a decreasing sequence of events and
$\event_\infty = \bigcap_{\nRuns\ge\laststart} \event_\nRuns$, by continuity from above we have
%\YGH{Do we have a name for this theorem/lemma?}
%\FI{Continuity from above}
%\PM{Nice, I was wondering about that myself! \smiley}
\begin{equation}
    \prob(\event_\infty) = \lim_{\nRuns\rightarrow\infty} \prob(\event_\nRuns) \ge 1-\smallproba,
\end{equation}
concluding the proof.
%\end{proofpart}

\item
%\YGH{I will write in $\run$ here as in the other convergence proofs though I am not sure whether this is better or not.}
Applying \autoref{lem:4points} \ref{lem:4points-b} gives
%The first two inequalities appearing in \eqref{eq:local_stoch_inclu_4points} are always true,
\begin{align}
    \norm{\update-\sol}^2
    & \le
    \norm{\current-\sol}^2
    - 2\current[\step]\product{\inter[\vecfieldstoch]}{\inter-\sol}
    \notag\\
    &\enspace
    + \current[\step]^2\norm{\inter[\vecfieldstoch] - \past[\vecfieldstoch]}^2
    - \norm{\inter - \current}^2
    \notag\\
    & \le
    \norm{\current-\sol}^2
    - 2\current[\step]\product{\vecfield(\inter)}{\inter-\sol}
    \notag\\
    &\enspace
    - 2\current[\step]\product{\inter[\noise]}{\inter-\sol}
    \notag\\
    &\enspace
    + 2\current[\step]^2
        (\norm{\inter[\vecfieldstoch]}^2+\norm{\past[\vecfieldstoch]}^2)
    - \norm{\inter - \current}^2.
    \label{eq:local_stoch_conv_4points}
\end{align}
Furthermore, for any realization in $\current[\event]$, $\inter\in\nhd$ so that $\product{\vecfield(\inter)}{\inter-\sol}\ge\strong\norm{\inter-\sol}^2$
and thus equation \eqref{eq:peg_stoch_strong_bound} holds, which allows us to write
\begin{align}
    \norm{\update-\sol}^2\one_{\current[\event]}
    & \le
    \norm{\current-\sol}^2 \one_{\current[\event]}
    - 2\current[\step]\product{\vecfield(\inter)}{\inter-\sol} \one_{\current[\event]}
    \notag\\
    &\enspace
    - 2\current[\step]\product{\inter[\noise]}{\inter-\sol} \one_{\current[\event]}
    \notag\\
    &\enspace
    + 2\current[\step]^2
        (\norm{\inter[\vecfieldstoch]}^2+\norm{\past[\vecfieldstoch]}^2) \one_{\current[\event]}
    - \norm{\inter - \current}^2 \one_{\current[\event]}
    \notag\\
    & \le
    (1-\strong\current[\step])
    \norm{\current-\sol}^2 \one_{\current[\event]}
    - 2\current[\step]\product{\inter[\noise]}{\inter-\sol} \one_{\current[\event]}
    \notag\\
    &\enspace
    + 2\current[\step]^2
        (\norm{\inter[\vecfieldstoch]}^2+\norm{\past[\vecfieldstoch]}^2) \one_{\current[\event]}
    + (2\strong\current[\step] - 1) \norm{\inter - \current}^2 \one_{\current[\event]}.
    \label{eq:local_stoch_conv_descent}
\end{align}
Similarly to
\eqref{eq:local_stoch_proba_v2bound}
and
\eqref{eq:local_stoch_proba_v2bound_bis},
we have
\begin{align}
    \ex[\norm{\inter[\vecfieldstoch]}^2 \one_{\current[\event]}]
    &\le
    \vbound^2 + \noisevar, \\
    \ex[\norm{\past[\vecfieldstoch]}^2 \one_{\current[\event]}]
    &\le
    \ex[\norm{\past[\vecfieldstoch]}^2 \one_{\last[\event]}]
    \le
    \vbound^2 + \noisevar.
\end{align}
We also recall that as $  \current[\event] \in \current[\filter]$, it holds
\begin{align}
    \ex[\product{\inter[\noise]}{\inter-\sol} \one_{\current[\event]}]
    = 
    \ex[
        \exof{\product{\inter[\noise]}{\inter-\sol} \given{\current[\filter]}}
        \one_{\current[\event]}
    ]
    = 0.
\end{align}
Taking expectation over \eqref{eq:local_stoch_conv_descent} then leads to
\begin{align}
    \ex[\norm{\update-\sol}^2 \one_{\current[\event]}]
    &\le 
    (1-\strong\current[\step])
    \ex[\norm{\current-\sol}^2 \one_{\current[\event]}]
    \notag\\
    &\enspace
    + 4\current[\step]^2 (\vbound^2 + \noisevar)
    + (2\strong\current[\step] - 1) \ex[\norm{\inter - \current}^2 \one_{\current[\event]}].
\end{align}
We can choose $\strongstepm$ sufficiently large so that
$2\strong\current[\step]-1\le0$
for all $\run\ge\start$.
Using $\current[\event]\subseteq\last[\event]$, we obtain
\begin{equation}
    \ex[\norm{\update-\sol}^2 \one_{\current[\event]}]
    \le 
    (1-\strong\current[\step])
    \ex[\norm{\current-\sol}^2 \one_{\last[\event]}]
    + 4\current[\step]^2 (\vbound^2 + \noisevar).
\end{equation}
%This is similar to \eqref{eq:peg_stoch_rec}.
By applying \autoref{lem:chung1954}
with
$\seqitem_{\run}
\subs
\ex[\norm{\current-\sol}^2 \one_{\last[\event]}]$,
$\consc\subs\strong\step$, %$p\subs1$,
$\conscalt\subs4\step^2(\vbound^2+\noisevar)$, and $\run_0\subs 1$,
we get
\begin{equation}
    \ex[\norm{\current[\state][\run]-\sol}^2 \one_{\last[\event][\run]}]
    \le \frac{4\step^2(\vbound^2+\noisevar)}
    {\strong\step-1}\frac{1}{\run} + \smalloh\left(\frac{1}{\run}\right).
\end{equation}
Finally,
\begin{align}
    \exof{\norm{\state_\run-\sol}^2\given{\event_\infty}}
    &= \frac{\ex[\norm{\current[\state][\run]-\sol}^2  \one_{\event_\infty}]}{\prob(\event_\infty)}
    \notag\\
    &\le
    \frac{\ex[\norm{\current[\state][\run]-\sol}^2 \one_{\last[\event][\run]}]}{1-\smallproba}
    \notag\\
    &\le
    \frac{4\step^2(\vbound^2+\noisevar)}{(\strong\step-1)(1-\smallproba)}\frac{1}{\run}
    + \smalloh\left(\frac{1}{\run}\right)
\end{align}
\end{enumerate}
and our proof is complete.
\end{proof}

\begin{remark*}
We notice that to complete the above proof, we only require Eq.~\eqref{eq:mean} and Eq.~\eqref{eq:variance} to be held on the event $\{\inter\in\nhd\}$.
For example, in \eqref{eq:local_stoch_proba_xi2bound} we want
$\ex[\exof{\norm{\inter[\noise][\run]}^2\given{\current[\filter][\run]}} \one_{\last[\eventalt][\run]}]\le\noisevar$ which is true
if Eq.~\eqref{eq:variance} holds on $\{\inter\in\nhd\}$ since $\last[\eventalt][\run]\subseteq\current[\event][\run]\subseteq\{\inter\in\nhd\}$.
This assumption is much weaker and more sensible.
It in particular shows that to obtain local guarantee we indeed only need the noise to be bounded locally. %in a neighborhood of $\sol$.
\end{remark*}